\documentclass[11pt, reqno]{amsart}
\usepackage[all]{xy}
\usepackage[T1]{fontenc}
\usepackage[utf8]{inputenc}
\usepackage{amsmath}
\usepackage{relsize,amsfonts}
\usepackage{amssymb}
\usepackage{bbm}
\usepackage{graphicx}
\usepackage{enumerate}
\usepackage{enumitem}
\usepackage{extarrows}
\usepackage{lmodern}
\usepackage{mathabx}
\usepackage{mathtools}
\usepackage{hyperref}
\usepackage{cleveref}
\usepackage{xcolor}
\numberwithin{equation}{section}
\usepackage[top=1.25in, bottom=1.25in, left=1in, right=1in]{geometry}
\usepackage[
backend=biber,
style=alphabetic,
]{biblatex}
\hypersetup{
    colorlinks=true,       
    linkcolor=blue,          
    citecolor=magenta,        
    filecolor=magenta,      
    urlcolor=cyan           
}

\newtheorem{thm}{Theorem}[section]

\newtheorem{prop}[thm]{Proposition}

\newtheorem{defi}[thm]{Definition}

\newtheorem{rmk}[thm]{Remark}

\newtheorem{lem}[thm]{Lemma}
\crefname{Lemma}{Lemma}{Lemmas}

\newcommand{\PP}{\mathbb{P}}
\newcommand{\ZZ}{\mathbb{Z}}
\newcommand{\QQ}{\mathbb{Q}}
\newcommand{\RR}{\mathbb{R}}
\newcommand{\CC}{\mathbb{C}}
\newcommand{\NN}{\mathbb{N}}
\newcommand{\EE}{\mathbb{E}}

\newcommand{\TT}{\mathbb{T}}

\newcommand{\ti}{\textit}

\newcommand{\Cl}{\text{Cl}}

\newcommand{\mfa}{\mathfrak{a}}

\newcommand{\mfb}{\mathfrak{b}}
\newcommand{\mfd}{\mathfrak{d}}
\newcommand{\mfm}{\mathfrak{m}}
\newcommand{\mfn}{\mathfrak{n}}
\newcommand{\mfp}{\mathfrak{p}}

\newcommand{\Mod}{\text{Mod}}

\newcommand\norm[1]{\left\lVert#1\right\rVert}
\makeatletter
\newcommand{\vertiii}[1]{{\left\vert\kern-0.25ex\left\vert\kern-0.25ex\left\vert #1 
    \right\vert\kern-0.25ex\right\vert\kern-0.25ex\right\vert}}
\newcommand{\colim@}[2]{%
  \vtop{\m@th\ialign{##\cr
    \hfil$#1\operator@font colim$\hfil\cr
    \noalign{\nointerlineskip\kern1.5\ex@}#2\cr
    \noalign{\nointerlineskip\kern-\ex@}\cr}}%
}
\newcommand{\colim}{%
  \mathop{\mathpalette\colim@{\rightarrowfill@\scriptscriptstyle}}\nmlimits@
}
\renewcommand{\varprojlim}{%
  \mathop{\mathpalette\varlim@{\leftarrowfill@\scriptscriptstyle}}\nmlimits@
}
\renewcommand{\varinjlim}{%
  \mathop{\mathpalette\varlim@{\rightarrowfill@\scriptscriptstyle}}\nmlimits@
}

\makeatother

\begin{document}
\pagenumbering{arabic}
\colorlet{refkey}{blue!30}

\title{Pointwise ergodic theorem along primes of the form $x^2 + ny^2$}
\author{Jan Fornal}
\address{School of Mathematics, University of Bristol, Bristol BS8 1UG, UK}
\email{nc24166@bristol.ac.uk}
\date{\today}
\maketitle

\begin{abstract}
This paper resolves the question of pointwise convergence for ergodic averages of a single function along the set of polynomial values of primes of the form $x^2 + ny^2$. Following the influential paper of Bourgain \cite{bourgain1989pointwise}, we employ the Hardy-Littlewood circle method where major arc and minor arc estimates for the set of prime ideals constitute the main novelty of the paper. We also prove that our convergence results cannot be extended to class of $L^1$ functions. 
\end{abstract}

\tableofcontents 

\section{Introduction} \label{section-1}

Pointwise ergodic theory dates back to the work of Birkhoff who proved the pointwise ergodic theorem for ordinary ergodic averages. Motivated by a question of Furstenberg, Bourgain started investigating subsequential ergodic averages. These take the form:
\begin{equation}
  \frac{1}{M} \sum_{k \leq M} f(T^{a_k} x).
\end{equation}
More precisely, he published in Publications Math\'ematiques de l'IH\'ES \cite{bourgain1989pointwise} the highly creative and celebrated proof of the following:
\begin{thm} \label{ergodic-theorem-polynomial-images}
  Let $P$ be a polynomial with integer coefficients. Fix $q > 1$, a measure preserving system $(X, \mu, T)$. If $f \in L^q(X)$, then:
  \begin{equation}
    \frac{1}{m} \sum_{k \leq m} f(T^{P(k)} x)
  \end{equation}
  converges for almost all $x \in X$.
\end{thm}
In 1988, building off work of Bourgain, specifically the $L^2$ formulation of the below result, M\'at\'e Wierdl in his paper \cite{wierdl1988pointwise} managed to prove the following:

\begin{thm} \label{ergodic-theorem-prime-number-case}
  Let $\PP$ be the set of prime numbers. Fix $q > 1$ and suppose that $(X, \mu, T)$ is measure-preserving system. If $f \in L^q(X)$, then:
  \begin{equation}
    \frac{1}{|\{p \in \PP : p \leq m\}|} \sum_{p \leq m} f(T^p x)
  \end{equation}
  converges for almost all $x \in X$.
\end{thm}

In the previous two theorems, by a measure-preserving system, we mean a $\sigma$-finite measure space equipped with a measure-preserving transformation; throughout $(X, \mu, T)$ will denote such a measure-preserving system. These two papers made a significant impact in the field of pointwise ergodic theory and in the author's opinion they changed the general perception of this area of analysis. Later on, mathematicians gave plenty of modifications regarding convergence of nonstandard ergodic averages. 
Our focus is the class of primes: 
\begin{equation}
  \PP_n = \{p \in \PP : \text{ there exists } u, v \in \ZZ \text{ so that } p = u^2 + nv^2 \},
\end{equation}
where $n$ is fixed for the rest of the paper. These primes have drawn the attention of various eminent mathematicians such as Fermat, Euler, Gauss, Hecke and Hilbert. They comply with various interesting principles coming from both analytic number theory and algebraic number theory. Using these primes, we will construct new types of ergodic averages and the main goal of this paper is to fully address the issue of pointwise convergence of these averages. The first result is qualitative:

\begin{thm} \label{ergodic-theorem-special-prime-numbers}
  Let $n$ be positive integer. Fix $q > 1$, a measure preserving system $(X, \mu, T)$ and a polynomial $P(x) \in \ZZ[x]$. If $f \in L^q(X)$, then:
  \begin{equation}
    A_m^n f(x) := A_m f(x) = \frac{1}{|\{p \in \PP_n : p \leq m\}|} \sum_{\substack{p \in \PP_n \\ p \leq m}} f(T^{P(p)} x)
  \end{equation}  
  converges for almost all $x \in X$.
\end{thm}

We further address quantitative convergence statistics for these averages, namely $r$-variation and the jump counting function, introduced by L\'epingle in the context of martingales (\cite{lepingle1976variation}) and imported to the ergodic-theoretic context by Bourgain (see for instance \cite{bourgain2006approach}); we also address Bourgain's oscillation seminorm. We will introduce the precise definitions of $r$-variation, the jump counting function and the oscillation seminorm in Section \ref{section-7}. Much work has been devoted to addressing the interplay between these operators and classical questions in harmonic analysis and ergodic theory (see for instance \cite{campbell2000oscillation}, \cite{zorin2015variation} or \cite{krause2022discrete}). Our next result proves that ergodic averages along primes of the form $x^2 + ny^2$ contribute another example to the aforementioned collection:
\begin{thm} \label{variational-ergodic-theorem-special-prime-numbers}
  Let $n$ be a positive integer. Fix $r > 2, q > 1$, a polynomial $P(x) \in \ZZ[x]$, a finite increasing sequence $I = (I_j)_{j \in [1, M]}$ and a positive real number $\lambda$. For all measurable functions $f$ on a measure-preserving system $(X, \mu, T)$, the following inequalities are satisfied:
  \begin{equation} \label{variation-inequality}
    \norm{\mathcal{V}^r \Bigl( A_m f(x) : m \in \ZZ \Bigr)}_{L^q(X)} \lesssim_{P, q, n} \frac{r}{r-2} \norm{f}_{L^q(X)} 
  \end{equation}
  \begin{equation} \label{oscillation-seminorm-inequality}
    \norm{\text{Osc}_I \Bigl( A_m f(x) : m \in \ZZ \Bigr)}_{L^q(X)} \lesssim_{P, q, n} \norm{f}_{L^q(X)} 
  \end{equation}
  \begin{equation} \label{jump-counting-inequality}
    \norm{\lambda N_\lambda^{1/2} \Bigl( A_m f(x) : m \in \ZZ \Bigr)}_{L^q(X)} \lesssim_{P, q, n} \norm{f}_{L^q(X)} 
  \end{equation}
  It is worthwhile mentioning explicitly that in the second and third inequalities, the implied constants does not depend on the choice of sequence $I$ nor the number $\lambda$, respectively.
\end{thm}

\begin{rmk}
  Lemma $2.12$ from \cite{mirek2020jumpinterpolation} together with interpolation methods justifies why \eqref{variation-inequality} follows from \eqref{jump-counting-inequality}. However, we will give the unified proof of Theorem \ref{variational-ergodic-theorem-special-prime-numbers} for all operators above.
\end{rmk}
The final result of our work addresses the $L^1$-endpoint. For the most popular nonconventional ergodic averages i.e.~along primes and polynomial images, convergence was established only for functions in $L^p(X)$ with $p > 1$, in contrast to Birkhoff's theorem. For a long time it was a serious question to investigate convergence only under the assumption that $f \in L^1$. In celebrated work, Buczolich and Mauldin (in \cite{buczolich2010divergent}) established the \emph{divergence} for averages of the form:
\begin{equation} \label{divergent-average}
  \frac{1}{N} \sum_{n = 1}^N f(T^{n^d} x)
\end{equation} 
for general $f \in L^1(X)$ and $d = 2$. They used Sawyer's principle (more precisely, equation $(1)$ from \cite{sawyer1966maximal}) regarding the link between weak-type maximal inequalities and almost everywhere convergence/divergence of ergodic averages for functions in $L^1(X)$. Later on, LaVictoire invented a more general construction for showing that certain ergodic averages diverge (see Theorem 3.1 in \cite{lavictoire2011universally}), using it to show that ergodic averages along primes and ergodic averages diverge \eqref{divergent-average} for $d \geq 3$. We will use that construction to justify the following:
\begin{thm} \label{divergence-theorem-special-prime-numbers}
  Let $n$ be a positive integer. For every measure preserving system $(X, \mu, T)$, there exists a function $f \in L^1(X)$ so that:
  \begin{equation}
    A_m f(x) = \frac{1}{|\{p \in \PP_n : p \leq m\}|} \sum_{\substack{p \in \PP_n \\ p \leq m}} f(T^p x)
  \end{equation}  
  diverges for almost every $x \in X$.
\end{thm}

In particular, the theme of this paper is to provide additional input to the program of study of nonconventional ergodic averages introduced by Bourgain:
\begin{enumerate}
  \item Bourgain's investigations on ergodic averages with polynomial values appeared in a sequence of papers beginning with \cite{bourgain2006approach} and culminating with the famous \cite{bourgain1989pointwise}.
  \item The case of primes was resolved by Wierdl using simpler methods (see \cite{wierdl1988pointwise}).
  \item Weighted ergodic averages were investigated by Cuny and Weber in \cite{cuny2017ergodic}.
  \item A new paradigm was introduced in last decade by Mirek, Stein and collaborators using Ionescu-Wainger theory (respectively brought up in Section \ref{section-7}). In the series of papers \cite{mirek2019lp}, \cite{mireklp} they established estimates for variation and square function of multidimensional polynomial ergodic averages. 
  \item Variational estimates for primes are covered in \cite{zorin2021maximal} and \cite{mirek2017variational}. 
  \item Estimates for the jump counting function and oscillation seminorms for multidimensional ergodic averages for polynomial values/primes are discussed by Mirek, Stein, Zorin-Kranich and by Mehlhop and Słomian in \cite{mirek2020jump}, \cite{mehlhop2024oscillation}.
  \item Notions of oscillation proved to be particularly useful in \cite{bourgain2023multi}, giving substantial progress on the linear version of the Furstenberg-Bergelson-Leibman conjecture (i.e.~Conjecture $1.22.$ from the aforementioned paper).
\end{enumerate}
The strategy for establishing quantitative and qualitative results in pointwise ergodic theory goes back to Bourgain, as he provided an adaptation of the \textit{analytic-number-theoretic Hardy-Littlewood circle method} to accommodate the needs of pointwise ergodic theory. By the Cald\'eron's transference principle, one may focus on the integer system $(\ZZ, x \to x-1)$. In that context, the average with fixed scale can be rewritten as convolution with appropriate kernel. For instance, the expression on the integer system:
\begin{equation}
  \frac{1}{m} \sum_{k = 1}^m f(T^k x) = \frac{1}{m} \sum_{k = 1}^m f(x - k)
\end{equation}
can be rewritten as:
\begin{equation}
  \Bigl( \frac{1}{m} \sum_{k = 1}^m \delta_k \Bigr) \ast f(x)
\end{equation}
In limit (i.e.~when $m \to \infty$), these convolution kernels are convergent on the Fourier side to the Kronecker delta function, $\delta_0: \TT \to \CC$, that evaluates to one on zero and zero everywhere else. In the case of polynomial averages:
\begin{equation}
  \frac{1}{m} \sum_{k = 1}^m f(T^{P(k)} x) = \frac{1}{m} \sum_{k = 1}^m f(x - P(k))
\end{equation}
the analogous limit in Fourier space is equal to:
\begin{equation}
  \sum_{\substack{\frac{a}{q} \in \QQ \pmod 1 \\ (a, q) = 1}} \Biggl( \frac{1}{q} \sum_{b = 1}^q e \Bigl(\frac{aP(b)}{q} \Bigr) \Biggr) \delta_{a/q}
\end{equation}
For averages over primes weighted by the von Mangoldt function:
\begin{equation}
  \frac{1}{m} \sum_{p \in \PP \cap [m]} \Lambda(p) f(T^p x) = \frac{1}{m} \sum_{p \in \PP \cap [m]} \log p f(x - p)
\end{equation}
these limits become:
\begin{equation}
  \sum_{\substack{\frac{a}{q} \in \QQ \pmod 1 \\ (a, q) = 1}} \frac{\mu(q)}{\varphi(q)} \delta_{a/q},
\end{equation}
where the quotient between M\"obius function and Euler totient function $\frac{\mu(q)}{\varphi(q)}$ arises after using Ramanujan formula for exponential sum of the form:
\begin{equation}
  \frac{1}{\varphi(q)} \sum_{b \in [q] : (b, q) = 1} e \Bigl(\frac{b}{q} \Bigr)
\end{equation}
For clarity, we recall that: 
\begin{equation}
  \varphi(a) := |\{ b \in [a] : (b, a) = 1\}|
\end{equation}
and
\begin{equation}
  \mu(a) :=  \begin{cases}
    (-1)^{\text{number of prime divisors of } a} & \text{when } a \text{ is squarefree} \\ 
    0 & \text{ otherwise }
  \end{cases}
\end{equation}
In the simplest context we consider, i.e.~for averages:
\begin{equation}
  \frac{1}{m} \sum_{p \in \PP_n \cap [m]} \log p \cdot f(T^p x) = \frac{1}{m} \sum_{p \in \PP_n \cap [m]} \log p \cdot f(x - p),
\end{equation}
the behaviour of the spectrum looks almost the same: it is given by a linear combination of Kronecker delta functions along rational primes with coefficients that are slightly more sophisticated exponential sums:
\begin{equation}
  \sum_{\substack{\frac{a}{q} \in \QQ \pmod 1 \\ (a, q) = 1}} \frac{S(a, q)}{2 |\Cl(\QQ(\sqrt{-n}))| \varphi_2(q_0)} \delta_{a/q},
\end{equation}
where $S(a, q), \varphi_2$ and $q_0$ are defined later (see respectively \eqref{weyl-sum}, \eqref{varphi-function} and \eqref{q_0-expression}) and $\Cl(\QQ(\sqrt{-n}))$ is the ideal class group of the field $\QQ(\sqrt{-n})$ (see Appendix \ref{appendix} for details).

As the limit spectrum is supported on the set of all rationals, it was necessary for Bourgain to provide additional tools to tackle his problem, beyond those which give a distinguished role to the zero frequency, namely Bourgain's multifrequency lemma and Bourgain's superorthogonality method. 
Naturally, Theorem \ref{variational-ergodic-theorem-special-prime-numbers} is strictly stronger than Theorem \ref{ergodic-theorem-special-prime-numbers}. However, we still intend to cover both of these approaches.
In our context, the main significant challenge that arises is the investigation of the Fourier transform of the characteristic function of primes $\PP_n \cap [N]$ when $N$ goes to infinity. Crucially important in analysing these sums are the formulas derived from the theory of Hecke L-functions and Vaughan's identity for general number fields. 
Some of the number-theoretic computations are far simpler in the case when $n \equiv 1, 2 \pmod 4$ and $n$ is squarefree. These are the cases when $n_0$ (introduced later) is equal to $1$. Nevertheless, we didn't want number theoretic complications to prevent us from fully resolving the above-raised questions for arbitrary positive $n$, so we accordingly cover all cases. \\

One can ask a similar question for arbitrary number fields: i.e.~with a fixed number field $K$, determine whether a similar result holds for the set of primes $\PP_K \subset \PP$ that split completely in $K$:
\begin{equation}
  \frac{1}{|\PP_K \cap [m]|} \sum_{p \in \PP_K \cap [m]} f(T^p x).
\end{equation}
The answer is positive for Galois extensions over $\QQ$ (although the Galois condition is only used in equations \eqref{prime-ideals-vs-prime-passing-before} and \eqref{prime-ideals-vs-prime-passing}) and the proof is extremely analogous to the proof contained in this paper.

\subsection{Structure of paper} The structure of the paper is as follows. In sections \ref{section-2}, \ref{section-3} and \ref{section-4} we provide all necessary number-theoretic background; in sections \ref{section-5} and \ref{section-6} we establish Theorem \ref{ergodic-theorem-special-prime-numbers}. 
Eventually, using machinery introduced in \cite{ionescu2006L}, we will resolve Theorem \ref{variational-ergodic-theorem-special-prime-numbers}. The last and very short section is devoted to the proof of Theorem \ref{divergence-theorem-special-prime-numbers}. For the convenience of those less familiar with algebraic number theory, we have attached a dictionary of the most important terms from that field in the Appendix \ref{appendix}.
\subsection{Acknowledgements} The author is grateful to Ben Krause for suggesting this problem, many fruitful discussions and assistance with typing up this article. Words of recognition also go to Hamed Mousavi and Joni Ter\"av\"ainen for discussions of Vaughan's identity for general number fields and Type II sums. Thanks also go to Tanja Eisner and Oleksiy Klurman for encouragement. 

\section{Number theory tools} \label{section-2}
\subsection{Number theory notation} 
Across the paper, we will use the following notation:
\begin{enumerate}
  \item $I_{\QQ(\sqrt{-n})}$ represents the group of all fractional ideals in $\QQ(\sqrt{-n})$. 
  \item Fractional ideals that are coprime to $q$ form the group $I_{\QQ(\sqrt{-n})}(q)$.
  \item Similarly, we introduce $\mathcal{P}_{\QQ(\sqrt{-n})}$ and $\mathcal{P}_{\QQ(\sqrt{-n})}(q)$ for the groups of principal ideals and principal ideals that are coprime to $q$, respectively. 
  \item We define $n_0$ as a quantity depending on the choice of $n$, namely if we factorize $n$ into squarefree and squareful parts $n = ab^2$, we define: 
  \begin{equation}
    n_0 = \begin{cases}
      2b &\text{ when } a \equiv 3 \pmod 4 \\
      b &\text{ otherwise}.
    \end{cases}
  \end{equation}
  Furthermore, let $\{1, \omega_n \}$ be an integral basis of $\mathcal{O}_{\QQ(\sqrt{-n})}$, the ring of integral elements of field $\QQ(\sqrt{-n})$. It's a standard exercise in algebraic number theory to show that:
  \begin{equation}
    \omega_n = \begin{cases}
      \frac{1 + \sqrt{-a}}{2} &\text{ when } a \equiv 3 \pmod 4 \\
      \sqrt{-a} &\text{ otherwise}.
    \end{cases}
  \end{equation}
  \item We provide the notation for the group of principal ideals that are generated by elements congruent to $1$ modulo $q$, $\mathcal{P}_{\QQ(\sqrt{-n})}^+(q)$; the notation for the group of principal ideals that are generated by elements written in the integral basis of $\mathcal{O}_{\QQ(\sqrt{-n})}$ with $\omega_n$-part divisible by $m$ is $\mathcal{P}_{\QQ(\sqrt{-n})}^m(q)$, where $m$ divides $q$.
  \item We will use notation $a \perp b$ to indicate that positive integers $a$ and $b$ are coprime.
  \item For $q > 1$, $\varphi_2(q)$ is defined as follows:
  \begin{equation} \label{varphi-function}
    \varphi_2(q) = \begin{cases}
      \frac{1}{2} \big|\{(u, v) \in [q]^2 : N(u + v \omega_n) \perp q \} \big| &\text{ if } q \neq 2 \\
      \big|\{(u, v) \in [q]^2 : N(u + v \omega_n) \perp q \} \big| &\text{ otherwise. } 
    \end{cases}
  \end{equation}
  This is precisely the size of quotient $\mathcal{P}_{\QQ(\sqrt{-n})}(q) / \mathcal{P}_{\QQ(\sqrt{-n})}^+(q)$ (the definition of the norm $N(\cdot)$ is recalled in Appendix \ref{appendix}). We also denote the variant without normalization by $\varphi_2'(q)$ i.e. 
  \begin{equation}
    \varphi_2'(q) = \big|\{(u, v) \in [q]^2 : N(u + v \omega_n) \perp q \} \big|.
  \end{equation}
  \item We will always denote $q_0$ by:
  \begin{equation} \label{q_0-expression}
    q_0 = qn_0
  \end{equation}
  \item For coprime $a, q$ with $a \leq q$, let $S(a, q)$ denote:
  \begin{equation} \label{weyl-sum}
    \begin{aligned}
      S(a, q) &:= \sum_{\mathfrak{a} \in \mathcal{P}_{\QQ(\sqrt{-n})}^{n_0}(q_0) / \mathcal{P}_{\QQ(\sqrt{-n})}^+(q_0)} e\Bigl(\frac{a P(N \mathfrak{a})}{q}\Bigr) \\
      &= w(q_0) \sum_{\substack{u \in [q_0], v \in [q_0] : N(u + v \omega_n) \perp q_0 \\ n_0 | v}} e\Bigl(\frac{a P(N(u + v \omega_n))}{q}\Bigr)
    \end{aligned}
  \end{equation}
  with $w(q_0)$ defined as:
  \begin{equation}
    w(q_0) := \begin{cases}
      \frac{1}{2} &\text{ if } q_0 > 2 \\
      1 &\text{ otherwise. } \\
    \end{cases}
  \end{equation}
  \item For ideals $\mathfrak{a} \in I_{\QQ(\sqrt{-n})}$ we can define the classical arithmetic functions as follows:
  \begin{enumerate}
    \item The divisor counting function $\tau$ evaluated at $\mathfrak{a}$ is follows:
    \begin{equation} \label{divisor-counting-function}
      \tau(\mathfrak{a}) = \sum_{\mfb \subset \mathcal{O}_{\QQ(\sqrt{-n})}: \mfb | \mfa} 1.
    \end{equation}
    \item The M\"obius function $\mu$ is defined as in the rational setting:
    \begin{equation} \label{mobius-function}
      \mu(\mfa) = \begin{cases}
        0 &\text{ if } \mfa \text{ is divisible by a square of a prime ideal } \\
        (-1)^{\text{number of prime ideals dividing } \mfa} &\text{ otherwise }. \\
      \end{cases}
    \end{equation} 
    \item The \textit{Von Mangoldt} function $\Lambda(\mfa)$ is $\log N \mfp$ when $\mfa$ is power of a prime ideal $\mfp$ and $0$ otherwise. 
    \item In this section as well as in the rest of the work we will write:
    \begin{equation}
      \begin{aligned}
        A &\lesssim B \\
        A &\lesssim_{p_1, \ldots, p_r} B \\
      \end{aligned}
    \end{equation}
    whenever there exists respectively a constant $C$ / a function $C(p_1, \ldots, p_r)$ such that:
    \begin{equation}
      \begin{aligned}
        A &\leq CB \\
        A &\leq C(p_1, \ldots, p_r) B \\
      \end{aligned}
    \end{equation}
  \end{enumerate}
  Equivalently, we will write that $A = O(B)$ and $A = O_{p_1, \ldots, p_r}(B)$ respectively. We will also use notation $A \ll B$ to indicate that $B$ is bigger than certain implied quantity depending on $A$.
\end{enumerate}

\subsection{Essential arithmetic tools}
The functions \eqref{varphi-function}, \eqref{weyl-sum} that were introduced above respect a few important properties:
\begin{lem} \label{arithmetic-functions-estimates}
  For every $\epsilon > 0, q$, $a$ that is coprime to $q$, and positive integer $k$, one has that:
  \begin{equation}
    \varphi_2(q) \gtrsim_{\epsilon} q^{2 - \epsilon},
  \end{equation}
  \begin{equation}
    S(a, q) \lesssim_{\epsilon, d, n} q^{2 - c_d + \epsilon},
  \end{equation}
  \begin{equation}
    \begin{aligned}
      & \frac{1}{\varphi_2'(q_0)} \sum_{\substack{u, v \in [q_0] : N(u + v \omega_n) \perp q_0 \\ n_0 | v}} e \Bigl(\frac{a P(N(u + v \omega_n))}{q} \Bigr) \\
      & \qquad \qquad \qquad \qquad \qquad \qquad = \frac{1}{\varphi_2'(kq_0)} \sum_{\substack{u, v \in [kq_0] : N(u + v \omega_n) \perp kq_0 \\ n_0 | v}} e \Bigl(\frac{a P(N(u + v \omega_n))}{q} \Bigr),
    \end{aligned}
  \end{equation}
  where $c_d$ is a certain constant depending on the degree of the polynomial $P$.
\end{lem}

\begin{proof}
  From the Chinese remainder theorem, all functions on the left-hand size are almost multiplicative / preserve multiplicative behaviour. The easiest to justify will be why the second inequality is true.
  It is enough to establish the same for $\varphi_2'(q)$ which is genuinely multiplicative (that is the variant of $\varphi_2$ without dividing by two in case when $q > 2$). We obtain:
  \begin{equation}
    \varphi_2'(q) \geq \varphi(q)^2
  \end{equation}
  from a couple of sentences:
  \begin{equation}
    \forall_{p \in \PP, \alpha > 1} \qquad \varphi_2'(p^\alpha) = \varphi_2'(p) p^{2 \alpha - 2} 
  \end{equation}
  \begin{equation} \label{varphi-caseology}
    \forall_{p \in \PP} \qquad \varphi_2'(p) \in \{ p^2 - 1, p(p-1), (p-1)^2 \}.
  \end{equation}
  Let's comment why the last inclusion is satisfied. Suppose that $p = 2$, then $\varphi_2'(p)$ expresses the number of solutions for one of four congruences:
  \begin{equation}
    u^2 + \epsilon_0 uv + \epsilon_1 v^2 \equiv 1 \pmod 2,
  \end{equation}
  where $(\epsilon_0, \epsilon_1) \in \{ 0, 1\}^2$. The number of solutions is therefore either $1, 2$ or $3$. In the case when $p > 2$, one is interested again about congruence:
  \begin{equation}
    u^2 + (\omega_n + \overline{\omega}_n) uv + \omega_n \overline{\omega}_n v^2 \not \equiv 0 \pmod p,
  \end{equation}
  or after a simple algebraic transformation:
  \begin{equation}
    u^2 + tv^2 \not \not \equiv 0 \pmod p
  \end{equation}
  for some $t \in \mathbb{F}_p$. If $t$ is equal to $0$, then naturally this has precisely $p(p-1)$ solutions. In the case when $-t$ is a non-quadratic residue modulo $p$, then $u^2 + tv^2$ is nonzero for precisely $p^2 - 1$ choices of $(u, v) \in [p]^2$. Finally, when $-t$ is a quadratic residue modulo $p$, then $u^2 + tv^2$ factorizes into a product of two linearly independent linear forms over $p$:
  \begin{equation}
    u^2 + tv^2 \equiv (u + rv)(u - rv) \pmod p
  \end{equation}
  for some $r \neq 0$; this is nonzero for $(p-1)^2$ pairs $(u, v) \in [p]^2$, finishing the proof of \eqref{varphi-caseology}.
  It is a bit harder to prove a similar fact for:
  \begin{equation}
    S(a, q)' := \sum_{\substack{u, v \in [qn_0] : N(u + v \omega_n) \perp qn_0 \\ n_0 | v}} e\Bigl(\frac{a P(N(u + v \omega_n))}{q}\Bigr).
  \end{equation}
  In the previous sum, we pass to arithmetic progressions with respect to $n_0$, so our goal reduces to showing that:
  \begin{equation}
    \sum_{u, v \in [q] : F(u, v) \perp q} e\Bigl(\frac{a Q(u, v)}{q}\Bigr) \lesssim_{d, \epsilon} q^{2 - c_d + \epsilon},
  \end{equation}
  where $F$ is a degree $2$ polynomial, monic with respect to the $u$ variable and $Q$ is a degree $2d$ polynomial, so that the $u^{2d}$-coefficient is of order $O_{n, d}(1)$. If we denote this coefficient by $r_{2d}$, due to the Chinese remainder theorem, we need only show that:
  \begin{equation}
    T(a, p^\alpha) := \sum_{u, v \in [p^\alpha] : F(u, v) \perp p} e \Bigl( \frac{a Q(u, v)}{p^\alpha} \Bigr) \lesssim (r_{2d}, p^\alpha)^{c_d} p^{(2 - c_d) \alpha}.
  \end{equation}
  After applying Fourier inversion:
  \begin{equation}
    \begin{aligned}
      \sum_{u, v \in [p^\alpha] : F(u, v) \perp p} e \Bigl( \frac{a Q(u, v)}{p^\alpha} \Bigr) &= \sum_{u, v \in [p^\alpha]} e \Bigl( \frac{a Q(u, v)}{p^\alpha} \Bigr) \mathbbm{1}_{F(u, v) \perp p} \\
      &= \sum_{u, v \in [p^\alpha]} e \Bigl( \frac{a Q(u, v)}{p^\alpha} \Bigr) \Biggl( \frac{p-1}{p} - \frac{1}{p} \sum_{w = 1}^{p-1} e \Bigl( \frac{w F(u, v)}{p} \Bigr) \Biggr), \\
    \end{aligned}
  \end{equation}
  and rearranging sums on the right-hand side, we obtain:
  \begin{equation}
    T(a, p^\alpha) = \frac{p - 1}{p} \sum_{u, v \in [p^\alpha]} e \Bigl( \frac{a Q(u, v)}{p^\alpha} \Bigr) - \frac{1}{p} \sum_{w = 1}^{p-1} \sum_{u, v \in [p^{\alpha}]} e \Bigl( \frac{a Q(u, v) + w p^{\alpha - 1} F(u, v)}{p^{\alpha}} \Bigr).
  \end{equation}
  We consider three cases: if $d$ is bigger than $1$, then the $u^{2d}$-coefficient of $a Q(u, v) + w p^{\alpha - 1} F(u, v)$ is $a r_{2d}$. Using Weyl's estimate, we can estimate all $u,v$-indexed sums above by:
  \begin{equation}
    p^\alpha (p^\alpha, r_{2d}) \Bigl( \frac{p^\alpha}{(p^\alpha, r_{2d})} \Bigr)^{1 - c_d} = (p^\alpha, r_{2d})^{c_d} p^{(2 - c_d) \alpha},
  \end{equation}
  as the first factor contributes to every single $v$, and $(p^\alpha, r_{2d})$ arises when we split the $u$-range into intervals of length $\frac{p^\alpha}{(p^\alpha, r_{2d})}$; this division is necessary for applying Weyl's estimate.
  The case when $d = 1$ and $p^{\alpha - 1}$ does not divide $r_{2d}$ can be considered in exactly the same fashion, as it is enough to verify for $d = 1$ and $p^{\alpha - 1} \nmid r_{2d}$ that:
  \begin{equation} \label{arithmetic-functions-estimates-2}
    \sum_{u, v \in [p^{\alpha}]} e \Bigl( \frac{a Q(u, v) + w p^{\alpha - 1} F(u, v)}{p^{\alpha}} \Bigr) \lesssim p^{2 \alpha - c_d \alpha} (r_{2d}, p^\alpha)^{c_d}
  \end{equation}
  for all $w \in [p-1]$ but one. We know that for all $w$ except one that we have:
  \begin{equation}
    (a r_{2d} + w p^{\alpha - 1}, p^\alpha) = p^{\alpha - 1}.
  \end{equation}
  Weyl's estimate in this case says that the left-hand side of \eqref{arithmetic-functions-estimates-2} is bounded by:
  \begin{equation}
    p^{2 \alpha - 1} p^{1 - c_d} = p^{2 \alpha - c_d} = (p^\alpha, p^{\alpha - 1})^{c_d} p^{2 \alpha - c_d \alpha} \leq p^{2 \alpha - c_d \alpha} (r_{2d}, p^\alpha)^{c_d}.
  \end{equation}
  It remains to justify \eqref{arithmetic-functions-estimates-2} for $w = 0$, however this is again the same reasoning as earlier.
  In order to prove the third identity, suppose that $k = k_1 k_2$ where every prime divisor of $k_1$ also divides $q_0$ and $k_2 \perp q_0$. Then:
  \begin{equation}
    \varphi_2'(kq_0) = k_1^2 \varphi_2'(k_2) \varphi_2'(q_0).
  \end{equation}
  Also, in the numerator of the identity, we can extract the $k_1^2$ factor:
  \begin{equation}
    \begin{aligned}
      \sum_{\substack{u, v \in [kq] : N(u + v \omega_n) \perp k_2 q_0 \\ n_0 | v}} & e \Bigl(\frac{a P(N(u + v \omega_n))}{q} \Bigr) \\
      &= \sum_{\substack{u, v \in [k_2 q_0] \\ N(u + v \omega_n) \perp k_2 q_0 \\ n_0 | v}} \sum_{0 < i, j < k_1} e \Bigl( \frac{a (P(N(u + ik_2q_0 + (v + jk_2q_0) \omega_n)))}{q} \Bigr) \\
      &= \sum_{\substack{u, v \in [k_2 q_0] \\ N(u + v \omega_n) \perp k_2 q_0 \\ n_0 | v}} \sum_{0 < i, j < k_1} e \Bigl( \frac{a (P(N(u + v \omega_n)))}{q} \Bigr) \\
      &= k_1^2 \sum_{\substack{u, v \in [k_2 q_0] \\ N(u + v \omega_n) \perp k_2 q_0 \\ n_0 | v}} e \Bigl( \frac{a (P(N(u + v \omega_n)))}{q} \Bigr). \\
    \end{aligned}
  \end{equation}
  Now, fix $u_0$ and $v_0$ in $[q_0]$ so that $N(u_0 + v_0 \omega_n)$ is coprime to $q_0$. In the right-hand side sum of the above equation, there are exactly $\varphi_2'(k_2)$ elements $u_1, v_1 \pmod{k_2}$ so that: $\newline N(u_0 + u_1 q_0 + (v_0 + v_1 q_0) \omega_n)$ is coprime to $k_2$. 
  Therefore we end up with:
  \begin{equation}
    \begin{aligned}
      k_1^2 \sum_{\substack{u, v \in [k_2 q_0] \\ N(u + v \omega_n) \perp k_2 q_0 \\ n_0 | v}} &e \Bigl( \frac{a P(N(u + v \omega_n))}{q} \Bigr)  \\
      &= k_1^2 \varphi_2'(k_2) \sum_{\substack{u, v \in [q_0] \\ N(u + v \omega_n) \perp q_0 \\ n_0 | v}} e \Bigl(\frac{a P(N(u + v \omega_n))}{q} \Bigr),
    \end{aligned}
  \end{equation}
  as desired.
\end{proof}

\begin{lem} \label{divisor-function-number-fields}
  One has the following inequalities as we sum the number theoretic functions along all small-norm ideals:
  \begin{equation}
    \sum_{\mfa : N \mfa \leq x} 1 \lesssim_n x 
  \end{equation}
  \begin{equation}
    \sum_{\mfa : N \mfa \leq x} \tau(\mfa)^k \lesssim_{n, k} x (\log x)^{2^k - 1}.
  \end{equation}
\end{lem}

\begin{proof}
  Naturally, all inequalities are still true for general number fields. These inequalities have well-known analogues for rational integers (see equation $1.80$ from \cite{iwaniec2021analytic}), however for completeness of exposition we will discuss the proof here. The first one was already established in \cite{marcus1977number} (Theorem $39$ there). Therefore, we may focus on showing that for every $k \in \NN$:
  \begin{equation}
    \sum_{\mfa : N \mfa \leq x} \tau(\mfa)^k \lesssim_k x (\log x)^{2^k - 1}.
  \end{equation}
  The goal is to prove this via induction. For $k = 0$ this has been already done. For the inductive step, using the standard inequality $\tau(\mfa \mfb) \leq \tau(\mfa) \tau(\mfb)$, we estimate:
  \begin{equation}
    \begin{aligned}
      \sum_{\mfa : N \mfa \leq x} \tau(\mfa)^k &\leq \sum_{\mfb : N \mfb \leq x} \Biggl( \sum_{\mfa : N \mfa \leq \frac{x}{N \mfb}} \tau(\mfa \mfb)^{k-1} \Biggr) \\
      &\leq \sum_{\mfb : N \mfb \leq x} \tau(\mfb)^{k-1} \Biggl( \sum_{\mfa : N \mfa \leq \frac{x}{N \mfb}} \tau(\mfa)^{k-1} \Biggr) \\
      &\lesssim_k \sum_{\mfb : N \mfb \leq x} \tau(\mfb)^{k-1} \frac{x}{N \mfb} \Bigl( \log \frac{x}{N \mfb} \Bigr)^{2^{k-1} - 1}.
    \end{aligned}
  \end{equation}
  Therefore it suffices to verify that:
  \begin{equation}
    \sum_{\mfb : N \mfb \leq x} \tau(\mfb)^{k-1} \frac{x}{N \mfb} \lesssim_k x (\log x)^{2^{k-1}}.
  \end{equation}
  Splitting the sum into dyadic intervals, we end up with:
  \begin{equation}
    \begin{aligned}
      \sum_{\mfb : N \mfb \leq x} \tau(\mfb)^{k-1} \frac{x}{N \mfb} &\leq \sum_{l \leq \log_2 x} \sum_{\mfb : N \mfb \in [\frac{x}{2^{l+1}}, \frac{x}{2^l}]} \tau(\mfb)^{k-1} 2^{l+1} \\
      &\lesssim_k \sum_{l \leq \log_2 x} \frac{x}{2^l} \Bigl( \log \frac{x}{2^l} \Bigr)^{2^{k-1} - 1} 2^{l+1} \\
      &\leq \sum_{l \leq \log_2 x} x (\log x)^{2^{k-1} - 1} \\
      &\lesssim x (\log x)^{2^{k-1}},
    \end{aligned}
  \end{equation}
  as desired.
\end{proof}

\section{Major arc analysis} \label{section-3}

Turning to our spectral analysis, suppose we are interested in $\alpha \in \TT$ that is in \textit{major arc}: there exists an irreducible fraction $\frac{a}{q}$ so that 
\begin{equation}
  |\alpha - \frac{a}{q}| \leq \frac{(\log x)^B}{x^d}
\end{equation}
with $q \leq (\log x)^B$. We will denote the set of these $\alpha$ as $\mathcal{M}_{a/q, B}$ and union of major arcs by:
\begin{equation} \label{major-arc-definition}
  \mathcal{M} := \mathcal{M}_B = \bigcup_{a/q : a \perp q, q \leq (\log x)^B} \mathcal{M}_{a/q, B},
\end{equation}
suppressing the $x$-dependence.
Our goal is to establish the following theorem:
\begin{thm} \label{major-arc-behaviour}
  With $\alpha, a, q, x$ as above, one has the following equality:
  \begin{equation}
    \sum_{p \in \PP_n \cap [x]} e(\alpha P(p)) \log p = \frac{x S(a, q)}{R_n \varphi_2(q_0)} \int_0^1 e \Bigl( (\alpha - \frac{a}{q}) P(xu) \Bigr) du + O_B(x \exp(-c \sqrt{\log x}))
  \end{equation}
  where $R_n = 2 |\Cl(\QQ(\sqrt{-n}))|$.
\end{thm}
We begin the proof by stating the formula which encodes information regarding twisted character sums from Hecke L-functions:

\begin{prop} \label{hecke-character-prime-number-theorem}
  For every Hecke character from $I_{\QQ(\sqrt{-n})} / \mathcal{P}_{\QQ(\sqrt{-n})}^+(q)$ to $S^1$ (see Appendix \ref{appendix} for definitions):
  \begin{equation} \label{hecke-character-prime-number-theorem-identity}
    \sum_{\mathfrak{p} : N \mathfrak{p} \leq x} \chi(\mathfrak{p}) \log N \mathfrak{p} = E(\chi) x + O_B(x \exp(-c \sqrt{\log x}))
  \end{equation}
  where $E(\chi)$ is $1$ when $\chi$ is a trivial character i.e.~the image of $\chi$ belongs to the set $\{0, 1\}$, and $0$ otherwise.
\end{prop}

This fact can be derived by combining Theorem 5.33, Theorem 5.35 from \cite{iwaniec2021analytic} together with the Siegel zero upper bound coming from \cite{fogels1963ausnahmenullstelle}. With this in mind, we may now move to the proof of Theorem \ref{major-arc-behaviour}.
\begin{proof}
  Summing \eqref{hecke-character-prime-number-theorem-identity} over all Hecke characters times $\frac{1}{|\Cl(\QQ(\sqrt{-n}))| \varphi_2(q)}$, and representatives $\mathfrak{a}$ corresponding to principal ideal classes $\mathcal{P}_{\QQ(\sqrt{-n})} / \mathcal{P}_{\QQ(\sqrt{-n})}^+(q)$ in $I_{\QQ(\sqrt{-n})} / \mathcal{P}_{\QQ(\sqrt{-n})}^+(q)$ twisted by $e(\frac{a N \mathfrak{a}}{q})$, we get (for the brevity reasons the $c$ constant in the exponential may vary from line to line):
  \begin{equation} \label{l-function-reduction}
    \begin{gathered}
      \sum_{\mathfrak{a} \in \mathcal{P}_{\QQ(\sqrt{-n})}^{n_0}(q_0) / \mathcal{P}_{\QQ(\sqrt{-n})}^+(q_0)} e\Bigl(\frac{a P(N \mathfrak{a})}{q}\Bigr) \sum_{\chi \in (I_{\QQ(\sqrt{-n})}(q_0) / \mathcal{P}_{\QQ(\sqrt{-n})}^+(q_0))^*} \sum_{\mfp : N \mfp \leq x} \frac{\chi(\mfp) \log N \mfp \bar{\chi}(\mfa)}{|\Cl(\QQ(\sqrt{-n}))| \varphi_2(q_0)} \\
      = \sum_{\mathfrak{a} \in \mathcal{P}_{\QQ(\sqrt{-n})}^{n_0}(q_0) / \mathcal{P}_{\QQ(\sqrt{-n})}^+(q_0)} e\Bigl(\frac{a P(N \mathfrak{a})}{q}\Bigr) \sum_{\chi \in (I_{\QQ(\sqrt{-n})}(q_0) / \mathcal{P}_{\QQ(\sqrt{-n})}^+(q_0))^*} \bar{\chi}(\mfa) \frac{(E(\chi) x + O_B(x \exp(-c \sqrt{\log x})))}{|\Cl(\QQ(\sqrt{-n}))| \varphi_2(q_0)} \\
      = \sum_{\mathfrak{a} \in \mathcal{P}_{\QQ(\sqrt{-n})}^{n_0}(q_0) / \mathcal{P}_{\QQ(\sqrt{-n})}^+(q_0)} \Biggl( e\Bigl(\frac{a P(N \mathfrak{a})}{q}\Bigr) \frac{x}{|\Cl(\QQ(\sqrt{-n}))| \varphi_2(q_0)} + O_B (x \exp(-c \sqrt{\log x})) \Biggr) \\
      = \sum_{\mathfrak{a} \in \mathcal{P}_{\QQ(\sqrt{-n})}^{n_0}(q_0) / \mathcal{P}_{\QQ(\sqrt{-n})}^+(q_0)} e\Bigl(\frac{a P(N \mathfrak{a})}{q}\Bigr) \frac{x}{|\Cl(\QQ(\sqrt{-n}))| \varphi_2(q_0)} + O_B (x \exp(-c \sqrt{\log x})),
    \end{gathered}
  \end{equation}
  since $q^2$ can be incorporated into $O(x \exp(-c \sqrt{\log x}))$ with no harm; above $G^\ast$ denotes dual group to group $G$. If we rearrange the first two sums in the left-hand side, we end up with an inner sum that can be reduced using the orthogonality of Hecke characters:
  \begin{equation} \label{orthogonality-reduction}
    \begin{gathered}
      \sum_{\mathfrak{a} \in \mathcal{P}_{\QQ(\sqrt{-n})}^{n_0}(q_0) / \mathcal{P}_{\QQ(\sqrt{-n})}^+(q_0)} e\Bigl(\frac{a P(N \mathfrak{a})}{q}\Bigr) \sum_{\mfp : N \mfp \leq x} \sum_{\chi \in (I_{\QQ(\sqrt{-n})}(qn_0) / \mathcal{P}_{\QQ(\sqrt{-n})}^+(q_0))^*} \frac{\chi(\mfp) \log N \mfp \bar{\chi}(\mfa)}{|\Cl(\QQ(\sqrt{-n}))| \varphi_2(q_0)} \\
      = \sum_{\mathfrak{a} \in \mathcal{P}_{\QQ(\sqrt{-n})}^{n_0}(q_0) / \mathcal{P}_{\QQ(\sqrt{-n})}^+(q_0)} e\Bigl(\frac{a P(N \mathfrak{a})}{q}\Bigr) \sum_{\mfp : N \mfp \leq x} \mathbbm{1}_{\mfp \equiv \mfa \pmod{\mathcal{P}_{\QQ(\sqrt{-n})}^+(q_0)}} \log N \mfp \\
      = \sum_{\substack{\mfp : N \mfp \leq x \\ \mfp \in \mathcal{P}_{\QQ(\sqrt{-n})}^{n_0}}} e \Bigl(\frac{a P(N \mfp)}{q} \Bigr) \log N \mfp + O_n(q),
    \end{gathered}
  \end{equation} 
  where last equality follows from the fact that double sum runs through the same set of prime ideals as the last sum (up to prime ideals containing $q_0$).
  Recalling that $S(a, q)$ is defined in \eqref{weyl-sum}, and combining formulas \eqref{l-function-reduction} and \eqref{orthogonality-reduction} yields:
  \begin{equation} \label{prime-ideals-vs-prime-passing-before}
    \sum_{\substack{\mfp : N \mfp \leq x \\ \mfp \in \mathcal{P}_{\QQ(\sqrt{-n})}^{n_0}}} e\Bigl(\frac{a P(N \mfp)}{q}\Bigr) \log N \mfp = S(a, q) \frac{x}{|\Cl(\QQ(\sqrt{-n}))| \varphi_2(q_0)} + O_B (x \exp(-c \sqrt{\log x})). \\
  \end{equation}
  We use the well-known fact that the prime ideals in $\mathcal{O}_{\QQ(\sqrt{-n})}$ either have norm $r^2$ for a rational prime $r$, which occurs only when $\mfp = (r)$, or $N \mfp$ is itself prime. If we sum the expression 
  \begin{equation}
    e \Bigl( \frac{a P(N \mfp)}{q} \Bigr) \log N \mfp 
  \end{equation}
  over prime ideals from the first case, we get the bound $x^{1/2} (\log x)$, which can be incorporated into $O_B (x \exp(-c \sqrt{\log x}))$. In the sum above, we are summing over prime ideals belonging $\mathcal{P}_{\QQ(\sqrt{-n})}^{n_0}$, therefore by the choice of $n_0$, one has that $N \mfp$ is of the form $x^2 + ny^2$. This correspondence acts in two ways, i.e.~every prime of the form $x^2 + ny^2$ can be written as the norm of a prime ideal $\mfp \in \mathcal{P}_{\QQ(\sqrt{-n})}^{n_0}$ for two different choices of $\mfp$. Therefore:
  \begin{equation} \label{prime-ideals-vs-prime-passing}
    \sum_{p \in \PP_n \cap [x]} e\Bigl(\frac{aP(p)}{q}\Bigr) \log p = S(a, q) \frac{x}{2 |\Cl(\QQ(\sqrt{-n}))| \varphi_2(q_0)} + O_B (x \exp(-c \sqrt{\log x})). \\
  \end{equation}
  If we denote $R_n := 2 |\Cl(\QQ(\sqrt{-n}))|$, then a standard summation by parts argument describes:
  \begin{equation}
    \sum_{p \in \PP_n \cap [x]} e (\alpha P(p)) \log p
  \end{equation}
  for $\alpha$ close to $\frac{a}{q}$. More precisely, we have:
  \begin{equation}
    \begin{gathered}
      \sum_{p \in \PP_n \cap [x]} e (\alpha P(p)) \log p - \frac{S(a, q)}{R_n \varphi_2(q_0)} \sum_{m \leq x} e \Bigl( (\alpha - \frac{a}{q}) P(m) \Bigr)  \\
      = \sum_{m \leq x} e \Bigl( (\alpha - \frac{a}{q}) P(m) \Bigr) \Biggl( e \Bigl( \frac{a P(m)}{q}\Bigr) \log m \mathbbm{1}_{m \in \PP_n} - \frac{S(a, q)}{R_n \varphi_2(q_0)} \Biggr) \\
      = O_B \Bigl(x (1 + x^d |\alpha - \frac{a}{q}|) \exp(-c \sqrt{\log x})\Bigr).
    \end{gathered}
  \end{equation}
  In particular, for $\alpha$ in the major arc surrounding $\frac{a}{q}$ i.e.~when $|\alpha - \frac{a}{q}| \leq \frac{(\log x)^B}{x^d}$, the error term is still of form $O_B(x \exp(-c \sqrt{\log x}))$. Eventually we obtain that:
  \begin{equation} \label{major-arc-behaviour-before}
    \sum_{p \in \PP_n \cap [x]} e(\alpha P(p)) \log p = \frac{S(a, q)}{R_n \varphi_2(q_0)} \sum_{m \leq x} e \Bigl( (\alpha - \frac{a}{q}) P(m) \Bigr) + O_B(x \exp(-c \sqrt{\log x})).
  \end{equation}
  We would now like to replace the exponential sum above with an appropriate oscillatory integral as in this case standard upper estimates on oscillatory integrals are more convenient than estimates on exponential sums. Observe that:
  \begin{equation}
    \bigg| e\Bigl( (\alpha - \frac{a}{q}) P(m) \Bigr) - \int_{m - 1}^m e \Bigl((\alpha - \frac{a}{q}) P(y) \Bigr) dy \bigg| \leq 2 \pi \bigg|\alpha - \frac{a}{q} \bigg| \sup_{y \in [m, m+1]} |P'(y)|
  \end{equation}
  so by the triangle inequality and the Riemann summation we derive the following:
  \begin{equation}
    \bigg| \sum_{m \leq x} e \Bigl( (\alpha - \frac{a}{q}) P(m) \Bigr) - \int_0^x e \Bigl((\alpha - \frac{a}{q}) P(y) \Bigr) dy \bigg| = O_d((\log x)^B),
  \end{equation}
  as $m \to \sup_{y \in [m-1, m]} |P'(y)| = O_P(m^{d - 1})$. Incorporating the above formula into \eqref{major-arc-behaviour-before} and changing variables, we obtain:
  \begin{equation} \label{major-arc-behaviour-2}
    \sum_{p \in \PP_n \cap [x]} e(\alpha P(p)) \log p = \frac{x S(a, q)}{R_n \varphi_2(q_0)} \int_0^1 e \Bigl( (\alpha - \frac{a}{q}) P(xu) \Bigr) du + O_B(x \exp(-c \sqrt{\log x})),
  \end{equation}
  finishing the proof of Theorem \ref{major-arc-behaviour}.
\end{proof}

\section{Minor arc analysis} \label{section-4}

Our goal in this section will be showing the following proposition:
\begin{thm} \label{minor-arc-inequality}
  For any $A \gg 1$ there exists bigger $B \gg A$, so that whenever $\alpha \in m_B := \TT \backslash \mathcal{M}_B$, we have: 
  \begin{equation}
    \sum_{\substack{\mfn : N \mfn \leq x \\ \mfn \in \mathcal{P}_{\QQ(\sqrt{-n})}^{n_0}}} \Lambda(\mfn) e(\alpha P(N \mfn)) \lesssim_n \frac{x}{(\log x)^A}.
  \end{equation}
\end{thm}
The argument from earlier \eqref{prime-ideals-vs-prime-passing} and the above inequality will suffice to yield:
\begin{equation} \label{minor-arc-inequality-rational-prime-version}
  \sum_{p \in \PP_n \cap [x]} \Lambda(p) e(\alpha P(p)) \lesssim_n \frac{x}{(\log x)^A}.
\end{equation}
Before we start discussing the proof of \ref{minor-arc-inequality}, let us recall Vaughan's identity for general number fields, which will be of crucial importance to the below line of reasoning:
\begin{equation} \label{vaughan-identity-number-fields}
  \begin{aligned}
    \Lambda(\mfn) &= \sum_{\substack{\mfd \mfa = \mfn \\ N \mfd \leq V}} \mu(\mfd) \log N \mfa - \sum_{\substack{\mfm \mfd \mfa = \mfn \\ N \mfm \leq U \\ N \mfd \leq V}} \Lambda(\mfm) \mu(\mfd) \\
    &+ \sum_{\substack{\mfm \mfd \mfa = \mfn \\ N \mfm > U \\ N \mfd > V}} \Lambda(\mfm) \mu(\mfd)
  \end{aligned}
\end{equation}
as long as the norm of $\mfn$ exceeds $U$. For context, we recall Vaughan's identity for integers, which presents as follows:
\begin{equation}
  \begin{aligned}
    \Lambda(n) &= \sum_{\substack{da = n \\ d \leq V}} \mu(d) \log a - \sum_{\substack{mda = n \\ m \leq U \\ d \leq V}} \Lambda(m) \mu(d) \\
    &+ \sum_{\substack{mda = n \\ m > U \\ d > V}} \Lambda(m) \mu(d)
  \end{aligned}
\end{equation}
whenever $n > U$. In the proof of Theorem \ref{minor-arc-inequality}, we will still require to interpret \eqref{vaughan-identity-number-fields} in the language of numbers instead of ideals. 
\begin{proof}[Proof of \ref{minor-arc-inequality}]
  For convenience we introduce the notation $\mathcal{P}_{\QQ(\sqrt{-n})}^{n_0} = \mathfrak{P}$. We twist the left-hand side of \eqref{vaughan-identity-number-fields} by $e(\alpha N \mfn)$ and sum over all principal ideals with norm at most $x$. Therefore:
  \begin{equation} \label{twisted-vaughan-identity}
    \begin{aligned}
      \sum_{\substack{\mfn : N \mfn \leq x \\ \mfn \in \mathfrak{P}}} \Lambda(\mfn) e(\alpha P(N \mfn)) &= \sum_{\substack{\mfn : N \mfn \leq x \\ \mfn \in \mathfrak{P}}} \sum_{\substack{\mfd \mfa = \mfn \\ N \mfd \leq V}} \mu(\mfd) \log (N \mfa) e(\alpha P(N \mfn)) - \sum_{\substack{\mfn : N \mfn \leq x \\ \mfn \in \mathfrak{P}}} \sum_{\substack{\mfm \mfd \mfa = \mfn \\ N \mfm \leq U \\ N \mfd \leq V}} \Lambda(\mfm) \mu(\mfd) e(\alpha P(N \mfn)) \\ 
      &+ \sum_{\substack{\mfn : N \mfn \leq x \\ \mfn \in \mathfrak{P}}} \sum_{\substack{\mfm \mfd \mfa = \mfn \\ N \mfm > U \\ N \mfd > V}} \Lambda(\mfm) \mu(\mfd) e(\alpha P(N \mfn)) + O(U) =: S_1 + S_2 + S_3 + O(U).
    \end{aligned}
  \end{equation}
  Set $U = V = x^{2/5}$; we show that every double sum on the right-hand side is of form:
  \begin{equation}
    \sum_{\substack{\mfa, \mfb : N \mfa \leq R \\ N \mfa \mfb \leq x \\ \mfa \mfb \in \mathfrak{P}}} x_\mfa y_\mfb e(\alpha P(N \mfa \mfb))
  \end{equation}
  where:
  \begin{enumerate}[label=(\Alph*)]
    \item \label{type-II-sum-constraints-a} For every two ideals $\mfa, \mfb \in I_{\QQ(\sqrt{-n})}$, one has that $|x_\mfa| \leq \tau(\mfa)$ and $|y_\mfb| \leq \log N \mfb$.
    \item \label{type-II-sum-constraints-b} $R$ does not exceed $x^{9 / 10}$.
  \end{enumerate}
  For the respective sums on the right-hand side of \eqref{twisted-vaughan-identity}, one can just make the following assignments:
  \begin{enumerate}
    \item In the case of $S_1$, take $x_\mfa = \mu(\mfa)$ and $y_\mfb = \log(N \mfb)$;
    \item In the case of $S_2$, take $x_\mfa = \mu(\mfa) \cdot \mathbbm{1}_{N \mfa < V}$ and $y_\mfb = \sum_{\mfm | \mfb : N \mfm \leq U} \Lambda(\mfm)$;
    \item In the case of $S_3$, take $x_\mfa = \sum_{\mfd | \mfa : N \mfd > V} \mu(\mfd)$ and $y_\mfb = \Lambda(\mfb) \cdot \mathbbm{1}_{N \mfb > U}$.
  \end{enumerate}
  Therefore; our focus moves to showing that under Properties \ref{type-II-sum-constraints-a} and \ref{type-II-sum-constraints-b} from above list we have:
  \begin{equation} \label{minor-arc-inequality-2}
    \bigg| \sum_{\substack{\mfa, \mfb : N \mfa \leq R \\ N \mfa \mfb \leq x \\ \mfa \mfb \in \mathfrak{P}}} x_\mfa y_\mfb e(\alpha P(N \mfa \mfb)) \bigg| \lesssim_n \frac{x}{(\log x)^A}.
  \end{equation}
  Fix a representative $\mfa_0 \in I_{\QQ(\sqrt{-n})} / \mathfrak{P}$ and its inverse $\mfb_0$, due to finiteness of the class group, we will prove only that:
  \begin{equation} \label{minor-arc-inequality-3}
    \bigg| \sum_{\substack{\mfa, \mfb : N \mfa \leq R \\ N \mfa \mfb \leq x \\ \mfa \equiv \mfa_0, \mfb \equiv \mfb_0 \pmod{\mathfrak{P}}}} x_\mfa y_\mfb e(\alpha P(N \mfa \mfb)) \bigg| \lesssim_n \frac{x}{(\log x)^A}.
  \end{equation} 
  Grouping ideals with the same norm:
  \begin{eqnarray}
    z_c &= \sum_{\mfa \equiv \mfa_0 \pmod{\mathfrak{P}} : N \mfa = c} x_\mfa, \\
    t_d &= \sum_{\mfb \equiv \mfb_0 \pmod{\mathfrak{P}} : N \mfb = d} y_\mfb,
  \end{eqnarray}
  we can reduce the task of proving \eqref{minor-arc-inequality-3} to estimating
  \begin{equation} \label{minor-arc-inequality-6}
    \bigg| \sum_{a, b \geq 0: a \leq R , b \leq S, ab \leq x } z_a t_b e(\alpha P(ab)) \bigg| \lesssim_n \frac{x}{(\log x)^A}, 
  \end{equation}
  where:
  \begin{enumerate}
    \item The numbers $z_a$ and $t_b$ respect the bounds:
    \begin{equation}
      |z_a| \lesssim_n \tau^2(a) \leq \tau^4(a) 
    \end{equation}
    and
    \begin{equation}
      |t_b| \lesssim_n (\log b + O_n(1)) \tau(b) \leq (\log b + O_n(1)) \tau^3(b). 
    \end{equation}
    These inequalities come from the fact that there are at most $2$ prime ideal factors occuring in factorization of ideal $(p) \subset \mathcal{O}_{\QQ(\sqrt{-n})}$ for $p \in \PP$. 
  \end{enumerate}
  So, suppose that $k$ is chosen so that $2^k \sim (\log x)^{A + 128}$; then the left-hand side expression under \eqref{minor-arc-inequality-6} obeys:
  \begin{equation}
    \begin{aligned}
      \bigg| \sum_{\substack{a, b \geq 0: a \leq R \\ ab \leq x }} z_a t_b e(\alpha P_2(ab)) \bigg| &\leq \bigg| \sum_{\substack{a, b \geq 0: a \leq R \\ ab \leq x/2^k }} z_a t_b e(\alpha P_2(ab)) \bigg| + \\
      &+ \sum_{j = 1}^k \bigg| \sum_{\substack{a, b \geq 0: a \leq R \\ ab \in (x/2^j, x/2^{j-1}] }} z_a t_b e(\alpha P_2(ab)) \bigg|.
    \end{aligned}
  \end{equation}
  The first term on the right-hand side is $\frac{x}{(\log x)^{A + 1}}$ due to Lemma \ref{divisor-function-number-fields} in the field of rational numbers. On the other hand, if \eqref{minor-arc-inequality-6} is not satisfied, then by the pigeonhole principle, we extract some $j \leq k$ so that:
  \begin{equation}
    \bigg| \sum_{\substack{a, b \geq 0: a \leq R \\ ab \in (x/2^j, x/2^{j-1}] }} z_a t_b e(\alpha P_2(ab)) \bigg| \gtrsim_n \frac{x}{(\log x)^{A + 1}}.
  \end{equation}
  Due to further dyadic pigeonholing, there is $M \leq R$ such that:
  \begin{equation}
    \bigg| \sum_{\substack{a, b \geq 0: a \sim M \\ ab \in (x/2^j, x/2^{j-1}] }} z_a t_b e(\alpha P_2(ab)) \bigg| \gtrsim_n \frac{x}{(\log x)^{A + 2}}.
  \end{equation}
  With this lower bound in mind, Proposition 2.2 from \cite{matomaki2021discorrelation} is satisfied with $H = N = \frac{x}{2^j}$ and $\delta = \frac{1}{(\log x)^{A + 35}}$.
  In particular, if $r_d$ is the leading coefficient of $P_2$, then there exists $q \leq \delta^{-O_d(1)} \leq (\log x)^{O_{d, A}(1)}$, so that:
  \begin{equation}
    \norm{q d r_d \alpha}_{\RR / \ZZ} \leq \delta^{-O_d(1)} \frac{2^{jd}}{x^d} \leq \frac{(\log x)^{O_{d, A}(1)}}{x^d},
  \end{equation}
  as both $2^{jd}$ and $\delta^{-O_d(1)}$ have logarithmic size. Furthermore, $r_d$ is of order $O_{P, n}(1)$, so $\alpha$ must lie inside a major arc around a fraction with denominator of size $(\log x)^{O_{d, A}(1)}$ and $\frac{(\log x)^{O_{d, A}(1)}}{x}$ close to it. Putting $B \gg O_{d, A}(1)$ gives us contradiction with the fact that we are in minor arc. Therefore \eqref{minor-arc-inequality-6} and in consequence Theorem \ref{minor-arc-inequality} are both satisfied.
\end{proof}

\section{Proof of Theorem \ref{ergodic-theorem-special-prime-numbers} - first part} \label{section-5}

With these preliminaries in hand, we now move to showing Theorem \ref{ergodic-theorem-special-prime-numbers}. 
We discuss a few standard reductions:
\begin{enumerate}
  \item Adding weights to the $A_m$ averages i.e.
  \begin{equation}
    A_m' f(x) = \frac{1}{m} \sum_{p \leq m : p \in \PP_n} \log(p) f(T^{P(p)} x)
  \end{equation}
  will not affect the convergence theorem (see Lemma 1 from \cite{wierdl1988pointwise}). To get the sum of weights equal to $1$, we would need the weights to be multiplied by $R_n = 2|\Cl(\QQ(\sqrt{-n}))|$, however for clarity and correctness of the argument that gives no advantage.
  \item Bourgain in \cite{bourgain2006approach} gave a procedure to yield a pointwise convergence theorem like Theorem \ref{ergodic-theorem-prime-number-case} from a maximal ergodic inequality and orthogonality (he proved a weaker special case of the Theorem \ref{variational-ergodic-theorem-special-prime-numbers}). That procedure works in our case as well, therefore for now, we will just focus on showing that:
  \begin{equation} \label{lp-maximal-ergodic-inequality}
    \norm{\sup_m | A_m' f |}_{L^p(X)} \lesssim \norm{f}_{L^p(X)}.
  \end{equation}
  \item A standard transference principle of Calderon (see pages 86-88 from \cite{krause2022discrete}) says that the only system for which we need to prove \eqref{lp-maximal-ergodic-inequality} is $(\ZZ, S)$ where $S$ maps $x \in \ZZ$ to $x-1$.
  In the case of the integers, applying the weighted average operator $A_m'$ to the function $f$ is the same as convolving it with $K_m$ where:
  \begin{equation} \label{kernel-of-ergodic-average}
    K_m(x) = \frac{1}{m} \sum_{p \in \PP_n \cap [m]} \mathbbm{1}_{x = P(p)} \log p 
  \end{equation}
  \item We can restrict to the case $f \geq 0$, and restrict our set of times to powers of $2$, so it suffices to prove:
  \begin{equation}
    \norm{\sup_{m \in 2^{\NN}} \bigg| A_m' f \bigg|}_{l^p(\ZZ)} \lesssim \norm{f}_{l^p(\ZZ)}
  \end{equation} 
\end{enumerate}

Define now $L_m'$ and $v_m, \varphi_s$ as follows: 
\begin{equation} \label{approximant-definition}
  \begin{aligned}
    \widehat{L_m'}(\alpha) &= \sum_{2^s \leq (\log m)^B} \sum_{q \in [2^{s-1}, 2^s)} \sum_{a \in [q]: (a, q) = 1} \frac{S(a, q)}{R_n \varphi_2(q_0)} v_m \Bigl(\alpha - \frac{a}{q}\Bigr) \varphi_{6s} \Bigl(\alpha - \frac{a}{q} \Bigr), \\
    v_m(\alpha) &= \int_0^1 e ( \alpha P(mu) ) du, \\
    \varphi_s(\alpha) &= \varphi(2^s \alpha) \\
  \end{aligned}
\end{equation}
for a smooth mollifier $\varphi$ satisfying: 
\begin{equation}
  \mathbbm{1}_{[-1/4, 1/4]} \leq \varphi \leq \mathbbm{1}_{[-1/2, 1/2]},
\end{equation}
$R_n = 2 |\Cl(\QQ(\sqrt{-n}))|$ as above, and $B$ comes from our particular choice in Theorem \ref{minor-arc-inequality}.
We can group the fractions by denominator size to form operators $L_{m, s} : \ZZ \to \CC$ which respect:
\begin{enumerate}
  \item $L_m'$ is the sum of $(L_{m, s})_{s \leq B \log_2 \log m}$;
  \item On the Fourier side $L_{m, s}$ has the expansion:
  \begin{equation} \label{dyadic-part-kernel-approximant}
    \widehat{L_{m, s}}(\alpha) = \sum_{q \in [2^{s-1}, 2^s)} \sum_{a \in [q]: (a, q) = 1} \frac{S(a, q)}{R_n \varphi_2(q_0)} v_m \Bigl(\alpha - \frac{a}{q}\Bigr) \varphi_{6s} \Bigl(\alpha - \frac{a}{q} \Bigr).
  \end{equation}
  \item We also provide the operators $L_m$ \textit{built with larger amount of increments} i.e.:
  \begin{equation}
    L_m = \sum_{2^s \leq \sqrt{m}/16} L_{m, s}.
  \end{equation}
  The operators $L_m'$ are more convenient to use for the approximations in Fourier space we are about to introduce, but $(L_m)_{m \in \NN}$ will be more convenient for Bourgain's \textit{superorthogonality} approach.
\end{enumerate}

\subsection{Providing a Fourier approximant for the kernels $K_m$}
Our objective for this subsection is to prove that:
\begin{equation} \label{special-prime-kernel-and-approximant-on-fourier-side}
  \norm{\widehat{K_m} - \widehat{L_m'}}_{L^\infty(\TT)} \lesssim_{n, B} \frac{1}{(\log m)^A}.
\end{equation}

\begin{rmk} \label{special-prime-kernel-and-approximant-on-fourier-side-remark}
  By Lemma \ref{arithmetic-functions-estimates}, we have that for any $s$:
  \begin{equation}
    \norm{L_{m, s}}_{L^\infty(\TT)} \lesssim 2^{-c_d s}
  \end{equation}
  Therefore, the $L^\infty$-norm of the Fourier transform of $\widehat{K_m} - \widehat{L_m}$ will still be $O_{n, B}(\frac{1}{(\log m)^A})$ as long as \eqref{special-prime-kernel-and-approximant-on-fourier-side} is satisfied.
\end{rmk}

\begin{proof}
  We will mimic the proof of equation (22) from \cite{wierdl1988pointwise}, but will include the argument to jusitfy the advantage of manipulating the scaling of $\varphi_s$ as much as we wish.
  Without loss of generality, assume that $m$ is larger than some large constant depending on $B$.
  Let $\alpha \in \TT$; we need to distinguish between major and minor arcs.
  \begin{enumerate}
    \item If, for some irreducible fraction $\frac{a'}{q'}$, $\alpha$ lies in the major arc around $\frac{a'}{q'}$, then from Theorem \ref{major-arc-behaviour}, we obtain:
    \begin{equation} \label{special-prime-kernel-common-thing}
      \widehat{K_m}(\alpha) = \frac{S(a', q')}{R_n \varphi_2(q'_0)} v_m \Bigl( \alpha - \frac{a'}{q'}\Bigr) + O_B(\exp(-c \sqrt{\log m})).
    \end{equation}
    For any other irreducible fraction $\frac{a}{q}$ whose denominator is bounded by $(\log m)^B$, we have from $m \gg B$ that:
    \begin{equation}
      \bigg|\alpha - \frac{a}{q} \bigg| \geq \frac{1}{2 (\log m)^{2B}} 
    \end{equation}
    so consequently $v_m(\alpha - \frac{a}{q}) \lesssim \frac{(\log m)^{2B/d}}{m}$ (by applying Lemma B.2 from \cite{krause2022discrete}).
    Applying Lemma \ref{arithmetic-functions-estimates} we get that $\frac{S(a, q)}{\varphi_2(q_0)} \lesssim q^{-c_d}$. The supports of $\varphi_{6s}(\alpha - \frac{a}{q})$ when $q \in [2^{s-1}, 2^s)$ are disjoint so in effect:
    \begin{equation} \label{approximant-common-thing}
      \bigg| \widehat{L_m'}(\alpha) - \frac{S(a', q')}{R_n \varphi_2(q'_0)} v_m \Bigl( \alpha - \frac{a'}{q'}\Bigr) \varphi_{6s} \Bigl( \alpha - \frac{a'}{q'}\Bigr) \bigg| \lesssim \frac{(\log m)^{2B/d}}{m} \sum_{s = 1}^\infty \frac{1}{2^{c_d s}}.
    \end{equation}
    Now $\varphi_{6s}(\alpha - \frac{a'}{q'})$ is $1$, as by the size of $m$ we have:
    \begin{equation}
      \bigg| \alpha - \frac{a'}{q'} \bigg| < \frac{(\log m)^B}{m} < \frac{1}{4 \cdot 2^{6 B \log_2 \log m}} \leq \frac{1}{4 \cdot 2^{6s}}
    \end{equation}
    Combining \eqref{approximant-common-thing} and \eqref{special-prime-kernel-common-thing} yields \eqref{special-prime-kernel-and-approximant-on-fourier-side} in this case.
    \item If $\alpha$ lies in minor arc, then from \eqref{minor-arc-inequality-rational-prime-version}:
    \begin{equation}
      |\widehat{K_m}(\alpha)| \lesssim_n \frac{1}{(\log m)^A};
    \end{equation}
    the same argument as we did for \eqref{approximant-common-thing} shows that $\widehat{L_m'}$ is $O \Bigl(\frac{1}{(\log m)^B} \Bigr)$. To elaborate a bit more, $|\alpha - \frac{a}{q}|$ is always at least $\frac{(\log m)^B}{m^d}$ on the minor arc, so $v_m(\alpha - \frac{a}{q}) \lesssim \frac{1}{(\log m)^{B/d}}$, which is enough to make the argument from \eqref{approximant-common-thing} work. Obviously, the difference between $\widehat{K_m}$ and $\widehat{L_m'}$ also respects the polylogarithmic savings.
  \end{enumerate}
\end{proof}

\section{Bourgain's multi-frequency result revisited} \label{section-6}

In the case when $p = 2$, we will use a famous result of Bourgain for a maximal estimate on the family $(L_{m, s})_{m \in 2^\NN}$ with fixed $s$. The original version of that result states the following:
\begin{prop} \label{multi-frequency-operator-boundedness}
  Suppose that $\Theta := \{ \theta_1, \ldots, \theta_N \}$ is $1$-separated i.e.~$|\theta_i - \theta_j| > 1$ when $i$ is different than $j$ and $\widehat{\chi_k}(x) = \mathbf{1}_{[-1/2^{k+1}, 1/2^{k+1}]}(x)$.
  Then the maximal sublinear operator:
  \begin{equation}
    \mathcal{M}_\Theta : L^2(\RR) \ni f \to \sup_{k \geq 1} \Bigg|\sum_{n = 1}^N (\Mod_{\theta_n} \chi_k) \ast f \Bigg| \in L^2(\RR)
  \end{equation}
  has norm at most $(\log N)^2$ up to absolute constant.
\end{prop}
\begin{rmk}
  \begin{enumerate} 
    \item It is obvious that $\mathcal{M}_\Theta$ has norm at most $2N$ (use for instance the triangle inequality together with the Hardy-Littlewood maximal function); 
    \item Applying a simple dilation argument, this theorem works also when for any distinct indices $i, j$ one has $|\theta_i - \theta_j| < t$ and supremum in $\mathcal{M}_\Theta$ varies on range $k \geq \log_2 t + 1$.
    \item From Lemma 4.4 in \cite{bourgain1989pointwise} one can justify that $\mathcal{M}_\Theta$ is $l^2(\ZZ)$-bounded operator with norm $O((\log N)^2)$. 
  \end{enumerate}
\end{rmk}
The famous paper containing the proof of that proposition is \cite{bourgain1989pointwise} (see especially Lemma 4.1.)
The new result that we endeavour to prove is as follows:
\begin{prop} \label{multi-frequency-operator-boundedness-dyadic-fractions}
  For any $s \in \ZZ^+$ and $g \in l^2(\ZZ)$, one has the inequality:
  \begin{equation}
    \norm{\sup_{m \in 2^\NN : m \geq 4^{s+4}}  \bigg| \sum_{q \in [2^{s-1}, 2^s)} \sum_{a \in [q] : a \perp q} \int_\TT v_m \Bigl(\alpha - \frac{a}{q} \Bigr) \varphi_{6s} \Bigl( \alpha - \frac{a}{q} \Bigr) \hat{f}(\alpha) e(j \alpha) d \alpha \bigg|}_{l^2_j(\ZZ)} \lesssim s^2 \norm{f}_{l^2(\ZZ)}.
  \end{equation}
\end{prop}

\begin{rmk}
  Applying the above proposition for $g$ which on the Fourier side is given by:
  \begin{equation}
    \widehat{g}(\alpha) = \sum_{q \in [2^{s-1}, 2^s)} \sum_{a \in [q] : a \perp q} \frac{S(a, q)}{\varphi_2(q_0)} \varphi_{6(s-1)} \Bigl(\alpha - \frac{a}{q} \Bigr)\widehat{f}(\alpha),
  \end{equation}
  together with the estimates from Lemma \ref{arithmetic-functions-estimates}, gives for any $\epsilon > 0$:
  \begin{equation} \label{multi-frequency-operator-boundedness-dyadic-fractions-2}
    \norm{\sup_{m \in 2^\NN : m \geq 4^{s+4}}  \bigg| L_{m, s} \ast f \bigg|}_{l^2(\ZZ)} \lesssim_\epsilon 2^{(\epsilon - c_d) s} \norm{f}_{l^2(\ZZ)}.
  \end{equation}
  Combining this inequality together with \eqref{special-prime-kernel-and-approximant-on-fourier-side} gives an $l^2$-maximal inequality in \eqref{lp-maximal-ergodic-inequality}.
\end{rmk}

\begin{proof}
  For any two irreducible fractions $\frac{a_1}{q_1}$ and $\frac{a_2}{q_2}$ with denominators in $[2^{s-1}, 2^s)$, the supports of $\alpha \to \varphi_{6s} \Bigl( \alpha - \frac{a_1}{q_1} \Bigr)$ and $\alpha \to \varphi_{6s} \Bigl( \alpha - \frac{a_2}{q_2} \Bigr)$ are disjoint.
  The proof consists of two steps, in the first we will try to replace $v_m$ by $\chi_{\log_2 m}$. Due to a standard Parseval's-identity argument, the difference between the expressions:
  \begin{equation}
    \norm{\sup_{m \in 2^\NN : m \geq 4^{s+4}}  \bigg| \sum_{q \in [2^{s-1}, 2^s)} \sum_{a \in [q] : a \perp q} \int_\TT v_m \Bigl(\alpha - \frac{a}{q} \Bigr) \varphi_{6s} \Bigl( \alpha - \frac{a}{q} \Bigr) \hat{f}(\alpha) e(j \alpha) d \alpha \bigg|}_{l^2_j(\ZZ)} 
  \end{equation}
  and
  \begin{equation}
    \norm{\sup_{m \in 2^\NN : m \geq 4^{s+4}}  \bigg| \sum_{q \in [2^{s-1}, 2^s)} \sum_{a \in [q] : a \perp q} \int_\TT \chi_{\log_2 m} \Bigl(\alpha - \frac{a}{q} \Bigr) \varphi_{6s} \Bigl( \alpha - \frac{a}{q} \Bigr) \hat{f}(\alpha) e(j \alpha) d \alpha \bigg|}_{l^2_j(\ZZ)} 
  \end{equation}
  is bounded by:
  \begin{equation} \label{multi-frequency-operator-boundedness-dyadic-fractions-difference}
    \norm{\sum_{m \in 2^\NN : m \geq 4^{s+4}} \sum_{q \in [2^{s-1}, 2^s)} \sum_{a \in [q] : a \perp q} \Bigg| \Bigl( v_m \Bigl(\alpha - \frac{a}{q} \Bigr) - \chi_{\log_2 m} \Bigl(\alpha - \frac{a}{q} \Bigr) \Bigr) \varphi_{6s} \Bigl( \alpha - \frac{a}{q} \Bigr)\Bigg| }_{l^\infty(\TT)}.
  \end{equation}
  Take some $\alpha' \in \TT$ and suppose the nearest fraction to it is $\frac{a'}{q'}$. Suppose that distance between them is $\delta$:
  \begin{enumerate}
    \item Due to inequality $v_m(\beta) \lesssim \frac{1}{m |\beta|^{1/d}}$, the contribution from scales $m$ with $\frac{1}{m} < \delta^{1/d}$ is:
    \begin{equation}
      \sum_{k \geq 0} \frac{4}{2^k M \delta^{1/d}} = O \Bigl(\frac{1}{M \delta^{1/d}} \Bigr) = O(1)
    \end{equation}
    where $M \in 2^\NN$ is the smallest number satisfying $\frac{1}{M} < \delta^{1/d}$. We need not even consider $\chi_{\log_2 m}$ as they all vanish at $\alpha' - \frac{a'}{q'}$ as long as $m \geq M$.
    \item Around zero, $v_m$ is estimated via $|v_m(\beta) - 1| \leq O(m^d |\beta|)$ (the inequality follows straight from definition of $v_m$). The geometric series kicks in again from justification that the contribution from small $m$ is also $O(1)$.
  \end{enumerate}
  Consequently, \eqref{multi-frequency-operator-boundedness-dyadic-fractions-difference} is bounded by $O(\norm{f}_{l^2(\ZZ)})$, whereas the estimate
  \begin{equation}
    \norm{\sup_{m \in 2^\NN : m \geq 4^{s+4}}  \bigg| \sum_{q \in [2^{s-1}, 2^s)} \sum_{a \in [q] : a \perp q} \int_\TT \chi_{\log_2 m} \Bigl(\alpha - \frac{a}{q} \Bigr) \varphi_{6s} \Bigl( \alpha - \frac{a}{q} \Bigr) \hat{f}(\alpha) e(j \alpha) d \alpha \bigg|}_{l^2_j(\ZZ)} \lesssim s^2 \norm{f}_{l^2(\ZZ)}
  \end{equation}
  follows from the remark after Proposition \ref{multi-frequency-operator-boundedness}.
\end{proof}

\subsection{The High-Low method and superorthogonality}

In this final subsection, we will finish the proof of \eqref{lp-maximal-ergodic-inequality}, using the High-Low method. The approach is almost identical to that presented in Chapter 7 of \cite{krause2022discrete}, still running through entire argument will be appropriate to justify why we have not decided to select any simpler method. This method originates from work of Bourgain \cite{bourgain1989pointwise}.
For any $S > 1$, we will construct two operators $\mathcal{L}_S$ and $\mathcal{H}_S$ acting on functions $\ZZ \to \CC$ where:
\begin{equation} \label{high-low-method-conditions}
  \begin{gathered}
    \sup_{m \in 2^\NN} |K_m \ast f| \leq \mathcal{L}_S f + \mathcal{H}_S f \\
    \forall_{p > 1} \norm{\mathcal{L}_S f}_{l^p(\ZZ)} \lesssim_p S^2 \norm{f}_{l^p(\ZZ)} \\
    \norm{\mathcal{H}_S f}_{l^2(\ZZ)} \lesssim 2^{-\delta S} \norm{f}_{l^2(\ZZ)} \\
  \end{gathered}
\end{equation}
for an appropriate constant $\delta > 0$.

\begin{prop} \label{high-low-method-guarantees-maximal-lp-inequality}
  The inequalities from \eqref{high-low-method-conditions} are sufficient to establish the $l^p$-boundedness ($p \in (1, 2)$) of the operator:
  \begin{equation}
    f \to \sup_{m \in 2^{\NN}} |K_m \ast f|
  \end{equation}
\end{prop}

\begin{proof}
  Due to the Marcinkiewicz interpolation theorem (see Theorem 1.8 from \cite{krause2022discrete}), it is enough to check:
  \begin{equation}
    \bigg| \Bigl\{ \sup_{m \in 2^{\NN}} |K_m \ast \mathbbm{1}_E| \geq \lambda \Bigr\} \bigg| \lesssim_p \frac{|E|}{\lambda^p}.
  \end{equation}
  Whenever $\lambda \geq \frac{1}{100}$, we may use the already-established $l^2$-boundedness:
  \begin{equation}
    \bigg| \Bigl\{ \sup_{m \in 2^{\NN}} |K_m \ast \mathbbm{1}_E| \geq \lambda \Bigr\} \bigg| \lesssim \frac{|E|}{\lambda^2} \lesssim \frac{|E|}{\lambda^p}.
  \end{equation}
  In the other case, we need to wrestle a bit more. If we set $r = \frac{1 + p}{2}$, then:
  \begin{equation}
    \begin{aligned}
      \bigg| \Bigl\{ \sup_{m \in 2^{\NN}} |K_m \ast \mathbbm{1}_E| \geq \lambda \Bigr\} \bigg| &\leq \bigg| \Bigl\{ \mathcal{L}_S \mathbbm{1}_E \geq \frac{\lambda}{2} \Bigr\} \bigg| + \bigg| \Bigl\{ \mathcal{H}_S \mathbbm{1}_E \geq \frac{\lambda}{2} \Bigr\} \bigg| \\
      &\lesssim_p \Bigl( \frac{2}{\lambda} \Bigr)^r S^{2r} |E| + \Bigl( \frac{2}{\lambda} \Bigr)^2 2^{-2 \delta S} |E|. \\
    \end{aligned}
  \end{equation}
  Plugging $S = \frac{(2 - p) \log \frac{1}{\lambda}}{2 \delta}$, we bound both expressions by:
  \begin{eqnarray}
    \Bigl( \frac{2}{\lambda} \Bigr)^2 2^{-2 \delta S} &= \frac{4}{\lambda^p} \\
    \Bigl( \frac{2}{\lambda} \Bigr)^r S^{2r} &= \Bigl( \frac{2}{\lambda} \Bigr)^r \frac{(2 - p)^{2r} \log^{2r} \frac{1}{\lambda}}{(2 \delta)^{2r}} = O_{p, \delta}(\frac{1}{\lambda^p}),
  \end{eqnarray}
  as the logarithm grows slower than any power function.
\end{proof}

The key fact we will need is that the operators $K_m$ satisfy a \textit{partial $l^p$ maximal inequality}, similar to those coming from polynomial kernels:
\begin{lem} \label{lp-partial-maximal-inequality}
  For each $p > 1$:
  \begin{equation}
    \norm{\sup_{m \in 2^\NN : J \leq \log m < 2J} |K_m \ast f| }_{l^p(\ZZ)} \lesssim_p \log J \norm{f}_{l^p(\ZZ)}.
  \end{equation}
\end{lem}

A telescoping argument which does not exploit the behaviour of $K_m$ shows that above lemma is satisfied as long as we have: 
\begin{lem} \label{l2-boundedness-of-telescoping-operator}
  Suppose that $(m_i)_{i = 0}^r$ is an increasing sequence of powers of $2$ with all elements inside $[\exp(J)$, $\exp(2J)]$. Assume also that for each $i \in \{0, \ldots, r-1 \}$ one has that $\frac{m_{i+1}}{m_i} \geq J^{C_r}$.
  Then for any set of integers $F \subset \ZZ$ together with pairwise disjoint subsets $F_2, \ldots, F_r \subset F$ we have:
  \begin{equation}
    \norm{(K_{m_1} - K_{m_0}) \ast \Bigl( \prod_{i = 2}^r K_{m_i} \ast \mathbbm{1}_{F_i} \Bigr)}_{l^2(\ZZ)} \lesssim J^{-r} |F|^{1/2}.
  \end{equation}
\end{lem}

\begin{proof}
  The proof was given for polynomial values in \cite{krause2022discrete}, we just want to check that in our setting everything goes more or less the same:
  \begin{enumerate}
    \item We select parameters $A_0, (\mathbf{c}(i))_{i = 1}^r$ in the same way, here we also insist that the constant $A$ from Theorem \ref{minor-arc-inequality} is so large that $A \gg A_0$;
    \item In our case, the $(\Omega_{m_i})_{i = 0}^r$ operators are defined on the Fourier side as:
    \begin{equation} \label{telescoping-operator-replacement}
      \widehat{\Omega_{m_i}}(\alpha) = \sum_{q \leq J^{c(i)}} \sum_{a \in [q] : a \perp q} \frac{S(a, q)}{R_n \varphi_2(q_0)} v_{m_i} \Bigl(\alpha - \frac{a}{q} \Bigr) \varphi \Biggl( \frac{J^{A_0}}{3 m_i} \Bigl( \alpha - \frac{a}{q} \Bigr) \Biggr);
    \end{equation}
    \item The inequalities:
    \begin{equation} \label{telescoping-operator-replacement-2}
      \norm{\Omega_{m_i}}_{L^\infty(\TT)} \lesssim J^{2 c(i)} 
    \end{equation}
    and
    \begin{equation} \label{telescoping-operator-replacement-3}
      \norm{\widehat{K_{m_i}} - \widehat{\Omega_{m_i}}}_{L^\infty(\TT)} \lesssim J^{-c_d c(i)} 
    \end{equation}
    are still satisfied. In order to justify \eqref{telescoping-operator-replacement-2}, one sees that number of terms in the double sum in \eqref{telescoping-operator-replacement} is at most $J^{2 c(i)}$ and all of them have $l^1(\ZZ)$-norm on physical space equal $O(1)$.
    Meanwhile, for \eqref{telescoping-operator-replacement-3} we first notice that from Remark \ref{special-prime-kernel-and-approximant-on-fourier-side-remark}:
    \begin{equation}
      \norm{\widehat{K_{m_i}} - \widehat{L_{m_i}}}_{L^\infty(\TT)} \lesssim J^{-A} \lesssim J^{-c_d c(i)}.
    \end{equation}
    We now compare the $\widehat{L_{m_i}}$ with the $\widehat{\Omega_{m_i}}$:
    \begin{equation}
      \begin{aligned}
        \sum_{s: 2^s \leq J^{c(i)}} & \sum_{q \in [2^{s-1}, 2^s), a \in [q] : a \perp q} \frac{S(a, q)}{R_n \varphi_2(q_0)} v_{m_i} \Bigl(\alpha - \frac{a}{q} \Bigr) \Bigl( \varphi_{6s}(\alpha - \frac{a}{q}) - \varphi ( \frac{J^{A_0}}{3 m_i} (\alpha - \frac{a}{q}) ) \Bigr) \\
        \qquad \qquad &+ \sum_{s: J^{c(i)} \leq 2^s \leq m_i^{1/2}/16} \sum_{q \in [2^{s-1}, 2^s), a \in [q] : a \perp q}  \frac{S(a, q)}{R_n \varphi_2(q_0)} v_{m_i} \Bigl(\alpha - \frac{a}{q} \Bigr) \varphi_{6s}(\alpha - \frac{a}{q})      
      \end{aligned}
    \end{equation}
    By Lemma \ref{arithmetic-functions-estimates}, the second double sum is uniformly bounded by $\sum_{s : J^{c(i)} \leq 2^s \leq m_i^{1/2}/16} 2^{- c_d s}$, which is clearly affordable in \eqref{telescoping-operator-replacement-3}. For the first double sum we proceed as follows: terms within a single inner sum have disjoint supports, and whenever $\alpha$ lies in the support of the summand in the first sum accompanying a fraction $\frac{a}{q}$ with $q \in [2^{s-1}, 2^s)$, then:
    \begin{equation}
      |\alpha - \frac{a}{q}| \gtrsim \frac{m_i}{J^{A_0}}.
    \end{equation} 
    Using Lemma B.2. from \cite{krause2022discrete}, the oscillatory integral $v_{m_i}(\alpha - \frac{a}{q})$ is then $O(J^{-A_0/d}) = O(J^{-c_d c(i)})$ as we wanted.
    By arguing as in the Section 7 of \cite{krause2022discrete}, \eqref{telescoping-operator-replacement-2} and \eqref{telescoping-operator-replacement-3} allow us to reduce proof of \ref{l2-boundedness-of-telescoping-operator} to:
    \begin{equation} \label{telescoping-operator-replacement-4}
      \norm{(K_{m_1} - K_{m_0}) \ast \Bigl( \prod_{i = 2}^r \Omega_{m_i} \ast \mathbbm{1}_{F_i} \Bigr)}_{l^2(\ZZ)} \lesssim J^{-r} |F|^{1/2}.
    \end{equation}
    \item The advantage of the manoeuvre from the previous point is that in frequency space, $\prod_{i = 2}^r \Omega_{m_i} \ast \mathbbm{1}_{F_i}$ is supported on:
    \begin{equation}
      \Gamma = \bigcup_{\substack{a, q: a \leq q \leq J^{r c(r)} \\ a \perp q}} \Bigl\{\alpha : \bigg| \alpha - \frac{a}{q} \bigg| \lesssim \frac{J^{A_0}}{m_2^d} \Bigr\}
    \end{equation}
    On this set $K_{m_1}$ and $K_{m_0}$ are almost identical: from Theorem \ref{major-arc-behaviour} we express:
    \begin{equation}
      \widehat{K_{m_i}}(\alpha) = \frac{S(a, q)}{R_n \varphi_2(q_0)} v_{m_i}(\alpha - \frac{a}{q}) + O(\exp(-c \sqrt{J}))
    \end{equation}
    for $i = \{0, 1 \}$. The difference $K_{m_1} - K_{m_0}$ on the Fourier side can be absolutely bounded by:
    \begin{equation}
      \begin{aligned}
      & \qquad \frac{S(a, q)}{R_n \varphi_2(q_0)} \bigg| v_{m_1}(\alpha - \frac{a}{q}) - v_{m_0}(\alpha - \frac{a}{q}) \bigg| + O(\exp(-c \sqrt{J})) \\
      &\leq m_1^d |\alpha - \frac{a}{q}| + O(\exp(- c \sqrt{J})) \lesssim J^{A_0} \Bigl( \frac{m_1}{m_2} \Bigr)^d \leq J^{A_0 - dC_r} \leq J^{-A_0} 
      \end{aligned}
    \end{equation}
    due to Lemma B.2. from \cite{krause2022discrete}. On top of that, one has the obvious $l^1(\ZZ)$-control of $\widehat{\Omega_{m_i}}$ when $i \geq 2$:
    \begin{equation} \label{telescoping-operator-replacement-5}
      \norm{\Omega_{m_i}}_{l^1(\ZZ)} \lesssim J^{2 c(i)}
    \end{equation}
    so in consequence:
    \begin{equation} \label{telescoping-operator-replacement-6}
      \norm{\Omega_{m_i} \ast \mathbbm{1}_{F_i}}_{l^\infty(\ZZ)} \lesssim J^{2 c(i)}.
    \end{equation}
    Now we are ready to show \eqref{telescoping-operator-replacement-4}:
    \begin{equation}
      \begin{aligned}
        \norm{(K_{m_1} - K_{m_0}) \ast \Bigl( \prod_{i = 2}^r \Omega_{m_i} \ast \mathbbm{1}_{F_i} \Bigr)}_{l^2(\ZZ)} &= \norm{(\widehat{K_{m_1}} - \widehat{K_{m_0}}) \mathcal{F}_\ZZ \Bigl( \prod_{i = 2}^r \Omega_{m_i} \ast \mathbbm{1}_{F_i} \Bigr)}_{L^2(\TT)} \\
        &\leq \sup_{\alpha \in \Gamma} |\widehat{K_{m_1}} - \widehat{K_{m_0}}|(\alpha) \norm{\mathcal{F}_\ZZ \Bigl( \prod_{i = 2}^r \Omega_{m_i} \ast \mathbbm{1}_{F_i} \Bigr)}_{L^2(\TT)} \\
        &\lesssim J^{-A_0} \norm{\mathcal{F}_\ZZ \Bigl( \prod_{i = 2}^r \Omega_{m_i} \ast \mathbbm{1}_{F_i} \Bigr)}_{L^2(\TT)} \\
        &= J^{-A_0} \norm{\Bigl( \prod_{i = 2}^r \Omega_{m_i} \ast \mathbbm{1}_{F_i} \Bigr)}_{l^2(\ZZ)} \\
        &\lesssim J^{2c(2) + 2c(3) + \ldots + 2c(r-1) - A_0} \norm{\Omega_{m_r} \ast \mathbbm{1}_{F_r}}_{l^2(\ZZ)} \\
        &\leq J^{2c(2) + 2c(3) + \ldots + 2c(r) - A_0} \norm{\mathbbm{1}_{F_r}}_{l^2(\ZZ)} \\
        &\leq J^{2c(2) + 2c(3) + \ldots + 2c(r) - A_0} |F|^{1/2} \\
      \end{aligned}
    \end{equation}
    where in the third-to-last and second-to-last inequalities we are using respectively \eqref{telescoping-operator-replacement-6} and \eqref{telescoping-operator-replacement-2}. Due to choice of parameters $(c(i))_{i = 1}^r$ and the number $A_0$, we evenutally get the proof of Lemma \ref{l2-boundedness-of-telescoping-operator}.
  \end{enumerate}
\end{proof}

We will also need, at the end of the argument for \eqref{lp-maximal-ergodic-inequality}, the following sampling principle of Magyar-Stein-Wainger (introduced in two versions, respectively in Lemma $2.1.$ and Corollary $2.1$, from \cite{magyar2002discrete}):
\begin{prop} \label{sampling-principle-magyar-stein-wainger}
  Given a natural number $q$ and a bounded multiplier $\mathbf{m} : \RR \to B$ mapping to a finite-dimensional Banach space with support inside $[\frac{-1}{2q}, \frac{1}{2q})$, define the periodic multiplier:
  \begin{equation}
    \mathbf{m}_{per}^q(\alpha) = \sum_{n \in \ZZ} \mathbf{m}(\xi - \frac{a}{q}). 
  \end{equation}
  Then:
  \begin{equation}
    \norm{\mathbf{m}_{per}^q}_{M^p(\ZZ)} \lesssim \norm{\mathbf{m}}_{M^p(\RR)}.
  \end{equation}
  where the multiplier norms $M^p$ are defined as follows:
  \begin{equation}
    \norm{m}_{M^p(\mathbb{G})} = \sup_{f : \norm{f}_{L^p(\mathbb{G})} = 1} \norm{\mathcal{F}^{-1}_{\mathbb{G}}(m \cdot \mathcal{F}_{\mathbb{G}}(f))}_{L^p(\mathbb{G})}, \text{ where } \mathbb{G} = \mathbb{R} \text{ or } \mathbb{Z}.
  \end{equation}
\end{prop}

Eventually we may move to the proof of \eqref{lp-maximal-ergodic-inequality} by establishing the High-Low decomposition of the operator $f \to \sup_{m \in 2^\NN} |K_m \ast f|$. We will mimic the argument from \cite{krause2022discrete}.
Take some $C_0$ and $C$ with $1 \ll C_0 \ll C$ and separate the range of supremum into two intervals:
\begin{equation} \label{lp-maximal-inequality-splitting}
  \sup_{m \in 2^\NN} |K_m \ast f| \leq \sup_{m \in 2^\NN : m \leq 2^{2^{CS}}} |K_m \ast f| + \sup_{m \in 2^\NN : m > 2^{2^{CS}}} |K_m \ast f|.
\end{equation}
The first component can be attached to the low part by virtue of Lemma \ref{lp-partial-maximal-inequality}.
The second component is split further:
\begin{equation} \label{lp-maximal-inequality-splitting-2}
  \sup_{m \in 2^\NN : m > 2^{2^{CS}}} |K_m \ast f| \leq \sup_{m \in 2^\NN : m > 2^{2^{CS}}} |L_m \ast f| + \sum_{m \in 2^\NN : m > 2^{2^{CS}}} |K_m \ast f - L_m \ast f|
\end{equation}
The second sum has norm $O(2^{-S} \norm{f}_{l^2(\ZZ)})$ after applying \ref{special-prime-kernel-and-approximant-on-fourier-side-remark} for $A = 2$ and not-too-large $B$:
\begin{equation}
  \begin{aligned}
    \norm{\sum_{m \in 2^\NN : m > 2^{2^{CS}}} |K_m \ast f - L_m \ast f|}_{l^2(\ZZ)} &\leq \sum_{m > 2^{2^{CS}}} \norm{K_m \ast f - L_m \ast f}_{l^2(\ZZ)} \\
    &\leq \sum_{m \in 2^\NN : m > 2^{2^{CS}}} \norm{K_m \ast f - L_m \ast f}_{l^2(\ZZ)} \\
    &\lesssim_n \sum_{m \in 2^\NN : m > 2^{2^{CS}}} \frac{1}{(\log m)^2} \norm{f}_{l^2(\ZZ)} \\
    &\lesssim 2^{-CS} \norm{f}_{l^2(\ZZ)}. \\
  \end{aligned}
\end{equation}
Now we use the triangle inequality for the third time to address the first term in \eqref{lp-maximal-inequality-splitting-2}:
\begin{equation}
  \sup_{m \in 2^\NN : m > 2^{2^{CS}}} |L_m \ast f| \leq \sup_{m \in 2^\NN : m > 2^{2^{CS}}} \bigg|\sum_{s \leq S} L_{m, s} \ast f \bigg| + \sum_{s > S} \sup_{m \in 2^\NN : m > 4^{s + 4}} |L_{m, s} \ast f|.
\end{equation} 
Computations for the $l^2(\ZZ)$-norm of the second sum present as follows (we use \eqref{multi-frequency-operator-boundedness-dyadic-fractions-2}):
\begin{equation}
  \begin{aligned}
    \norm{\sum_{s > S} \sup_{m \in 2^\NN : m > 4^{s + 4}} |L_{m, s} \ast f|}_{l^2(\ZZ)} &\leq \sum_{s > S} \norm{\sup_{m \in 2^\NN : m > 4^{s+4}} |L_{m, s} \ast f|}_{l^2(\ZZ)} \\
    &\lesssim \sum_{s > S} 2^{-c_d s} \norm{f}_{l^2(\ZZ)} \\
    &\lesssim 2^{-c_d S} \norm{f}_{l^2(\ZZ)}. \\
  \end{aligned}
\end{equation}
It remains to decompose:
\begin{equation} \label{lp-maximal-inequality-splitting-3}
  \sup_{m \in 2^\NN : m > 2^{2^{CS}}} \bigg|\sum_{2^s \leq m} L_{m, s} \ast f \bigg| 
\end{equation}
Define $M_{m, S}$ and $Q_S$ as follows:
\begin{equation}
  Q_S = \text{lcm}(1, 2, \ldots, 2^S) 
\end{equation}
\begin{equation}
  \widehat{M_{m, S}}(\alpha) = \sum_{a \in [Q_S]} \frac{S(a, Q_S)}{R_n \varphi_2(Q_{S, 0})} v_m \Bigl(\alpha - \frac{a}{q} \Bigr) \varphi_{2^{C_0 S} + 2} \Bigl(\alpha - \frac{a}{q} \Bigr)
\end{equation}
We now conclude the rest of argument where we obtain the low-high part decomposition of \eqref{lp-maximal-inequality-splitting-3} from the Magyar-Stein-Wainger principle:
\begin{enumerate}
  \item The reason why $\widehat{M_{m, S}}$ is close to $\widehat{L_m}$ around fractions with denominator at most $\frac{\sqrt{m}}{16}$ is that the exponential sum $\frac{S(a, Q_S)}{\varphi_2(Q_{S, 0})}$ is invariant under reducing by $\text{gcd}(a, Q_S)$. This is a consequence of the last identity from Lemma \ref{arithmetic-functions-estimates}. We use the information that $|\alpha - a/q| \gtrsim 2^{-2^{C_0 S}}$ to bound:
  \begin{equation}
    \begin{aligned}
      \norm{\sup_{m \in 2^\NN : m > 2^{2^{CS}}} \bigg| \Bigl(\sum_{s \leq S} L_{m, s} - M_{m, S} \Bigr) \ast f \bigg| }_{l^2(\ZZ)} &\leq \sum_{m \in 2^\NN : m > 2^{2^{CS}}} \norm{ \Bigl(\sum_{s \leq S} L_{m, s} - M_{m, S} \Bigr) \ast f}_{l^2(\ZZ)} \\
      \leq \sum_{m \in 2^\NN : m > 2^{2^{CS}}} & \norm{ \Bigl(\sum_{s \leq S} L_{m, s} - M_{m, S} \Bigr) }_{L^\infty(\TT)} \norm{f}_{l^2(\ZZ)} \\
      \leq \sum_{m \in 2^\NN : m > 2^{2^{CS}}} & \frac{2^{2^{C_0 S} / d}}{m^{1/d}} \norm{f}_{l^2(\ZZ)}
    \end{aligned}
  \end{equation}
  where in the last inequality we have used Lemma B.2 from \cite{krause2022discrete}. The sum on the right-hand side respects exponential decay with respect to $S$, so $\sum_{s \leq S} L_{m, s} - M_{m, S}$ is subsumed in the low part.
  \item The fraction $\frac{S(a, Q_S)}{\varphi_2(Q_{S, 0})}$ is a convex combination of exponential phases with integral coefficients (they are $N(u + v \omega_n)$ for integral $u, v$). Therefore Proposition \ref{sampling-principle-magyar-stein-wainger} is applicable.
\end{enumerate}
This concludes the proof of \eqref{lp-maximal-ergodic-inequality}.

\section{Proof of Theorem \ref{variational-ergodic-theorem-special-prime-numbers}} \label{section-7}

For the next section we focus on the issue of quantifying convergence; we shall begin by introducing the framework for Ionescu-Wainger theory. 
\begin{defi} \label{ionescu-wainger-set}
  Let $\rho \in (0, 1)$ be a very small parameter (later on, we will fix this precisely) and assume that $N \geq 2^R$, where $R = \lfloor \frac{2}{\rho} \rfloor + 1$. Set:
  \begin{equation}
    S_\rho(N) = \Bigl\{ \prod_{p \leq N^{\rho / 2}} p^{\lfloor \log N / \log p \rfloor} : p \in \PP \Bigr\} \cup \bigcup_{p \in (N^{\rho/2}, N] \cap \PP} \Bigl\{ p^{\lfloor \log N / \log p \rfloor} \Bigr\}.
  \end{equation}
  The $N$-th set of Ionescu-Wainger frequencies with parameter $\rho$ is:
  \begin{equation}
    \begin{aligned}
      \Sigma_{\leq R}(N) &= \Bigl\{ \frac{a}{q} : a \leq q, \\
      &\text{ there are at most } R \text{ elements in } S_\rho(N) \text{ whose product is } q  \Bigr\}
    \end{aligned}
  \end{equation}
\end{defi}

Below we list all the properties of this set that we will need below:
\begin{enumerate}[label=(\Alph*)] 
  \item \label{ionescu-wainger-theory-recap-a} There exists an absolute constant $C_0$ for which the following statement is true: all fractions with denominator at most $N$ belong to $\Sigma_{\leq R}(N)$, and all elements of $\Sigma_{\leq R}(N)$ are fractions with denominator at most $C_0^{R^2 N^{\rho/2}}$.
  \item \label{ionescu-wainger-theory-recap-b} For a multiplier $\mathbf{m} : \RR \to \RR$ supported on sufficiently small interval (i.e.~$[-\exp(-N^\rho), \\ \exp(- N^\rho)]$), we introduce the $\leq N$-th Ionescu-Wainger multiplier with parameter $\rho$ as:
  \begin{equation} \label{ionescu-wainger-multiplier}
    \Pi_{\leq N}[\mathbf{m}](\beta) = \Pi_{\leq N}^\rho[\mathbf{m}](\beta) = \sum_{\theta \in \Sigma_{\leq R}(N)} \mathbf{m}(\beta - \theta).
  \end{equation}
  This multiplier satisfies the following estimate for any $p \in (1, \infty)$:
  \begin{equation} \label{ionescu-wainger-inequality}
    \norm{\Pi_{\leq N}[\mathbf{m}]}_{M^p(\ZZ)} \lesssim_{p, \rho} \norm{\mathbf{m}}_{M^p(\RR)},
  \end{equation}
  \item \label{ionescu-wainger-theory-recap-c} For fractions $a/q$ we will denote by $\mathbf{h}(a/q)$ its Ionescu-Wainger height, that is, the smallest dyadic $N$ so that $a/q \in \Sigma_{\leq R}(N)$ 
\end{enumerate}

Apart from that, we will use a few results coming from discrete harmonic analysis:
\begin{enumerate}
  \item The operators $\text{Osc}_{(I_j)_{j \in [1, N]}}(f_n(x) : n \in \NN)$, $\mathcal{V}^r(f_n(x) : n \in \NN)$ and $N_{\lambda}(f_n(x) : n \in \NN)$ denote the following:
  \begin{equation}
    \text{Osc}_{(I_j)_{j \in [1, N]}}(f_n(x) : n \in \NN) = \Bigl( \sum_{j = 1}^N \sup_{i \in [I_j, I_{j+1})} |f_i - f_{I_j}|^2(x) \Bigr)^{1/2}
  \end{equation}
  \begin{equation}
    \mathcal{V}^r(f_n(x) : n \in \NN) = \sup_{j; n_0 < \ldots < n_j \in \NN} \Bigl( \sum_{i = 1}^j |f_i - f_{i-1}|^r(x) \Bigr)^{1/r}
  \end{equation}
  \begin{equation}
    N_{\lambda}(f_n(x) : n \in \NN) = \sup \{j : \text{ there exists } n_0 < \ldots < n_j \in \NN : \forall_{i \in [j]} |f_{n_i}(x) - f_{n_{i-1}}(x)| > \lambda \} 
    \medskip
  \end{equation}
  where for each $n$, $f_n$ are complex-valued functions with a $\sigma$-finite measure space $X$ as their domain.
  For any subsequence $r > 2$, $I \in \NN$ and any $\lambda$ one has pointwise control:
  \begin{equation} \label{2-variation-majorizes-all-operators-in-f}
    \text{Osc}_I(f_n(x) : n \in \NN), \mathcal{V}^r(f_n(x) : n \in \NN), \lambda N_{\lambda}^{1/2}(f_n(x) : n \in \NN) \leq \mathcal{V}^2(f_n(x) : n \in \NN)
  \end{equation}
  Additionally, by Lemma 2.5 from \cite{mirek2020bootstrapping}, one has the so-called Rademacher-Menshov inequality:
  \begin{equation} \label{rademacher-menshov-inequality}
    \mathcal{V}^2(f_n(x) : n \in [0, 2^m - 1]) \leq \sqrt{2} \sum_{i = 0}^{m - 1} \Bigl( \sum_{j = 0}^{2^{m - i}} \big| f_{2^i (j + 1)} - f_{2^i j} \big|^2(x) \Bigr)^{1/2}.
  \end{equation}
  \item Let $\mathcal{F}$ represent the following set of operators that take sequences of functions $(f_n : X \to \CC)_{n \in \NN}$ into single complex-valued functions with domain $X$:
  \begin{equation}
    \begin{aligned}
      \mathcal{F} &:= \Bigl\{ \text{Osc}_I : I = (I_n)_{n \in [1, N]} \text{ is finite subsequence in } \NN \Bigr\}  \\
      &\cup \Bigl\{ \lambda N_\lambda^{1/2} : \lambda > 0 \Bigr\} \cup \Bigl\{ \mathcal{V}^r : r > 2 \Bigr\}.
    \end{aligned}
  \end{equation}
  For an operator $U \in \mathcal{F}$, define the \textit{r-factor}, $r(U)$, as follows:
  \begin{equation}
    r(U) := \begin{cases}
      1 &\text{ if } U = \text{Osc}_I \text{ or } U = \lambda N_\lambda^{1/2} \\
      \frac{r}{r - 2} &\text{ if } U = \mathcal{V}^r.
    \end{cases}
  \end{equation}
  There are two basic properties that operators in $\mathcal{F}$ uniformly satisfy, the first of them is the so-called $l^1$-control:
  \begin{equation} \label{l1-control-on-operators-in-f}
    U(f_n : n \in \NN)(x) \lesssim \sum_{n = 1}^\infty |f_n|(x)
  \end{equation}
  and the second one is the \textit{quasi-triangle inequality}, i.e.: whenever $U \in \mathcal{F}$ and $(f_n)_{n \geq 1}$ and $(g_n)_{n \geq 1}$ are functions from $\sigma$-finite measure space $X$ to $\CC$, then there exists $U_1 \in \mathcal{F}$, $r(U) = r(U_1)$ so that:
  \begin{equation} \label{quasi-triangle-inequality}
    U(f_n + g_n : n \in \NN)(x) \lesssim U_1(f_n : n \in \NN)(x) + U_1(g_n : n \in \NN)(x)
  \end{equation}
  Whenever $U = \text{Osc}_I$ or $\mathcal{V}^r$, one may take $U_1 = U$ and the implicit constant above is $1$, when $U$ is the jump counting function, we get:
  \begin{equation}
    \lambda N_\lambda^{1/2}(f_n + g_n : n \in \NN)(x) \leq \lambda N_{\lambda/2}^{1/2}(f_n : n \in \NN)(x) + \lambda N_{\lambda/2}^{1/2}(g_n : n \in \NN)(x).
  \end{equation}
  \item One can split the variation into a \textit{long} and a \textit{short} part with respect to an increasing sequence of integers $\mathcal{A} = (A_j)_{j = 1}^\infty$. The corresponding long and short variations are given respectively by:
  \begin{equation}
    \mathcal{V}^{r, L}_{\mathcal{A}}(f_i)(x) = \sup_{n_1 < \ldots < n_t} \Biggl( \sum_{k = 1}^{t-1} |f_{A_{n_{k+1}}} - f_{A_{n_k}}|^r \Biggr)^{1/r}(x)
  \end{equation}
  and
  \begin{equation}
    \mathcal{V}^{r, S}_{\mathcal{A}}(f_i)(x) = \Biggl( \sum_{n = 1}^\infty \Bigl( \sup_{A_n \leq n_1 < \ldots < n_t < A_{n+1}} \sum_{k = 1}^{t-1} |f_{n_{k+1}} - f_{n_k}|^r \Bigr) \Biggr)^{1/r}(x).
  \end{equation}
  These notions where introduced firstly in \cite{jones2008strong}. Zorin-Kranich in \cite{zorin2015variation} proved that (see formula (2.4)):
  \begin{equation} \label{variation-splitted}
    \mathcal{V}^r(f_i)(x) \leq \mathcal{V}^{r, L}_{\mathcal{A}}(f_i)(x) + 2 \mathcal{V}^{r, S}_{\mathcal{A}}(f_i)(x).
  \end{equation}
  \item One can establish a similar construction for the jump counting function. As in the previous point, let $\mathcal{A} = (A_j)_{j = 1}^\infty$ be an increasing sequence of integers, then one defines:
  \begin{equation}
    \begin{aligned}
      N_{\lambda, \mathcal{A}}^L(f_i)(x) = \sup \{ r \in \NN : \text{ there exist numbers } &s_1 < t_1 \leq \ldots \leq s_r < t_r \\
      &\text{ so that } \forall_{i \in [r]} |f_{A_{t_i}} - f_{A_{s_i}}| > \lambda \} 
    \end{aligned}
  \end{equation}
  We recall the statement of Lemma 1.3 from \cite{jones2008strong}, namely that for all $\rho \geq 1$, $\lambda$ and sequence of functions $(f_i : X \to \CC)_{i \in \NN}$:
  \begin{equation} \label{variation-splitted-2}
    \lambda N_{\lambda}(f_n)(x)^{1/\rho} \leq 9 \Bigl( \mathcal{V}^{\rho, S}_{\mathcal{A}}(f_i)(x) + \lambda N_{\lambda / 3, \mathcal{A}}^L(f_i)(x)^{1 / \rho} \Bigr).
  \end{equation}
  Meanwhile, for the oscillation operator $\text{Osc}_I$, we may define the long counterpart as:
  \begin{equation}
    \text{Osc}^L_{I, \mathcal{A}}(f_n(x) : n \in \NN) = \Bigl( \sum_{j = 1}^N \sup_{A_i \in [J_j, J_{j+1}) } |f_{A_i} - f_{J_j}|^2(x) \Bigr)^{1/2}
  \end{equation}
  where the sequence $(J_j)_{j = 1}^M$ is constructed from $(I_j)_{j = 1}^M$ and $\mathcal{A} = (A_i)_{i \in \NN}$ in the following fashion:
  \begin{equation}
    J_j = \max \{A_i : i \in \NN, A_i \leq I_j \}.
  \end{equation}
  Due to Cauchy-Schwarz and Minkowski's inequality, we may analogously bound:
  \begin{equation} \label{variation-splitted-3}
    \text{Osc}_I(f_n)(x) \leq 5 \Bigl( \mathcal{V}^{2, S}_{\mathcal{A}}(f_i)(x) + \text{Osc}_{I, \mathcal{A}}^L(f_i)(x) \Bigr).
  \end{equation}
  \item We merge the data encoded in inequalities \eqref{variation-splitted}, \eqref{variation-splitted-2} and \eqref{variation-splitted-3} together. Let $U \in \mathcal{F}$, then $U_\mathcal{A}^L$ is defined as:
  \begin{equation}
    U_\mathcal{A}^L = \begin{cases}
      \mathcal{V}_{\mathcal{A}}^{r, L} &\text{ if } U = \mathcal{V}^r \\
      \frac{\lambda}{3} (N_{\lambda/3, \mathcal{A}}^L)^{1/2} &\text{ if } U = \lambda N_\lambda^{1/2} \\
      \text{Osc}_{I, \mathcal{A}}^L &\text{ if } U = \text{Osc}_I. \\
    \end{cases}
  \end{equation}
  We have the following inequality for all $U \in \mathcal{F}$:
  \begin{equation} \label{variation-splitted-4}
    U(f_n)(x) \leq 27 \Bigl( \mathcal{V}^{2, S}_{\mathcal{A}}(f_i)(x) + U_{\mathcal{A}}^L(f_i)(x) \Bigr).
  \end{equation}
  \item All operators from $\mathcal{F}$ satisfy L\'epingle's inequality, which we recall. Suppose one is dealing with operators concerning the dyadic martingale i.e.~the family $(\EE_N)_{N \in \NN}$ so that $\EE_N f(x) = \frac{1}{2^N} \int_{y \in I_x} f(y)$, where $I_x$ is dyadic interval of length $2^N$ containing $x$. Then for all $p \in (1, \infty)$, the estimates: 
  \begin{equation}
    \begin{gathered}
      \norm{\mathcal{V}^r(\EE_N f(x) : N \in \ZZ)}_{L^p(\RR)} \lesssim_p \frac{r}{r - 2} \norm{f}_{L^p(\RR)} \\
      \forall_{\lambda > 0} \norm{\lambda N_\lambda^{1/2}(\EE_N f(x) : N \in \ZZ)}_{L^p(\RR)} \lesssim_p \norm{f}_{L^p(\RR)} \\
      \forall_{I \in \NN : |I| < \infty} \norm{\text{Osc}_I(\EE_N f(x) : N \in \ZZ)}_{L^p(\RR)} \lesssim_p \norm{f}_{L^p(\RR)}
    \end{gathered}
  \end{equation} 
  were established as follows: the first in \cite{lepingle1976variation}, the second jointly in \cite{bourgain1989pointwise} (see equation 3.5) and \cite{jones1998oscillation} and the last one in \cite{jones1998oscillation} (see Theorem 6.4).
  The goal for this case is to specify precisely the form of L\'epingle inequality that will be useful later. Namely, fix $t > 0$ and suppose that $M_t$ is an averaging operator on functions $\RR \to \CC$:
  \begin{equation}
    M_t f(x) = \frac{1}{t} \int_0^t f(x - P(t)).
  \end{equation}
  Using Theorems $2.14$ and $2.39$ from \cite{mirek2020bootstrapping} together with Lemma $2.12$ from \cite{mirek2020jumpinterpolation}, one gets the following inequalities, we will slightly abuse notation and refer to all of them as L\'epingle's inequality for polynomial averaging operators: for $p \in (1, \infty)$ the following estimates hold
  \begin{equation} \label{lepingle-inequality-radon-averaging-operator}
    \begin{gathered}
      \norm{\mathcal{V}^r(M_t f(x) : t > 0)}_{L^p(\RR)} \lesssim_p \frac{r}{r - 2} \norm{f}_{L^p(\RR)} \\
      \forall_{\lambda > 0} \norm{\lambda N_\lambda^{1/2}(M_t f(x) : t > 0)}_{L^p(\RR)} \lesssim_p \norm{f}_{L^p(\RR)} \\
      \forall_{I \in \NN : |I| < \infty} \norm{\text{Osc}_I(M_t f(x) : t > 0)}_{L^p(\RR)} \lesssim_p \norm{f}_{L^p(\RR)}
    \end{gathered}
  \end{equation}
\end{enumerate}

In all the inequalities from Theorem \ref{variational-ergodic-theorem-special-prime-numbers} we will begin by using partial summation as in the case of maximal inequality, we state the following proposition, which will be useful in these efforts:
\begin{prop} \label{weight-transference-for-operators-in-f}
  Let $(w_n)_{n \in \NN}$ and $(w_n')_{n \in \NN}$ be non-negative sequences satisfying one of the following conditions:
  \begin{enumerate}
    \item The sequence $\Bigl( \frac{w_n'}{w_n} \Bigr)_{n \in \NN}$ decreases monotonically;
    \item The sequence $\Bigl( \frac{w_n'}{w_n} \Bigr)_{n \in \NN}$ increases monotonically and:
    \begin{equation}
      C = \sup_{N \in \NN} \frac{W_N w_N'}{W_N' w_n} < \infty
    \end{equation}
    where $W_N = \sum_{n = 1}^N w_n$ and $W_N' = \sum_{n = 1}^N w_n'$. 
  \end{enumerate}
  Then, for any sequence $(a_n)_{n \in \NN}$ of complex numbers, any finite subsequence $I = (i_j)_{j \leq M}$ of natural numbers, any real number $\lambda > 0$, and any $r > 2$:
  \begin{equation} \label{weight-transference-for-operators-in-f-inequality}
    \begin{gathered}
      \mathcal{V}^r \Bigl( \sum_{n = 1}^N w_n a_n : N \in \NN \Bigr) \lesssim \mathcal{V}^r \Bigl( \sum_{n = 1}^N w_n' a_n : N \in \NN \Bigr) \\
      \lambda N_\lambda^{1/2} \Bigl( \sum_{n = 1}^N w_n a_n : N \in \NN \Bigr) \leq \sum_j C_j D_j \lambda N_{D_j \lambda}^{1/2} \Bigl( \sum_{n = 1}^N w_n' a_n : N \in \NN \Bigr) \\
      \text{Osc}_I \Bigl( \sum_{n = 1}^N w_n a_n : N \in \NN \Bigr) \leq \sum_j C_j \text{Osc}_{I_j} \Bigl( \sum_{n = 1}^N w_n' a_n : N \in \NN \Bigr) 
    \end{gathered}
  \end{equation}
  where:
  \begin{enumerate}
    \item The implied constant in first inequality in \eqref{weight-transference-for-operators-in-f-inequality} depends only on  $C$;
    \item In the second and third inequalities, the set of indexes $j$ used for the outer summation might be countably infinite, but the sum $\sum_j C_j$ is bounded by $O_C(1)$;
    \item The sequences $(I_j)_{j \in \NN}, (C_j)_{j \in \NN}$ and $(D_j)_{j \in \NN}$ are only determined by the sequences $(w_n)_{n \in \NN}, \newline (w_n')_{n \in \NN}$ and do not rely on the choice of the sequence $(a_n)_{n \in \NN}$.
  \end{enumerate}
\end{prop}
The first inequality from \eqref{weight-transference-for-operators-in-f-inequality} was already stated in Proposition 5.1 from \cite{mirek2017variational}. The second and the third inequalities can be proved in a similar manner, still we will discuss the proof of the third inequality:
\begin{proof}
  Following Lemma 2 from \cite{mirek2017variational}, we introduce a double-indexed sequence $(\lambda_n^k)_{n, k \in \NN}$ of non-negative real numbers such that for every $k$ one has:
  \begin{equation}
    \sum_{n = 1}^\infty \lambda_n^k = \Lambda < \infty.
  \end{equation}
  Furthermore assume that for every $N \in \NN$, the sequence:
  \begin{equation}
    k \to \sum_{n = 1}^N \lambda_n^k
  \end{equation}
  is decreasing. We will firstly show that there are finite sequences $(I_j)_{j \in \NN}$ together with constants $C_j$ with $\sum_j C_j = \Lambda$ so that:
  \begin{equation} \label{weight-transference-for-operators-in-f-inequality-2}
    \text{Osc}_I \Bigl( \sum_{n = 1}^\infty \lambda_n^k a_n : k \in \NN \Bigr) \leq \sum_j C_j \text{Osc}_{I_j} \Bigl( a_n : n \in \NN \Bigr) 
  \end{equation}
  For each $k \in \NN$ we define a function $N_k : [0, \Lambda] \to \NN$ by:
  \begin{equation}
    N_k(t) = \inf \{ N \in \NN : \sum_{i = 1}^N \lambda_i^k > t\}
  \end{equation}
  and $I_n^k = \{t \in [0, \Lambda] : N_k(t) = n \}$. Following page 14 of \cite{mirek2017variational}, one gets that:
  \begin{equation} \label{weight-transference-identity}
    \sum_{n = 1}^\infty \lambda_n^k a_n = \int_0^\Lambda a_{N_k(t)} dt.
  \end{equation}
  Using this identity, we compute:
  \begin{equation} \label{weight-transference-for-operators-in-f-inequality-3}
    \begin{aligned}
      \Bigl( \sum_{j = 1}^{M-1} \sup_{k_j \in [i_j, i_{j+1})}\bigg| \sum_{n = 1}^\infty (\lambda_n^{k_j} - \lambda_n^{i_j}) a_n \bigg|^2 \Bigr)^{1/2} &= \Bigl( \sum_{j = 1}^{M-1} \sup_{k_j \in [i_j, i_{j+1})}\bigg| \int_0^\Lambda a_{N_{k_j}(t)} - a_{N_{i_j}(t)} dt \bigg|^2 \Bigr)^{1/2}  \\
      &= \Bigl( \sum_{j = 1}^{M-1} \Bigl( \int_0^\Lambda \sup_{k_j \in [i_j, i_{j+1})} \big| a_{N_{k_j}(t)} - a_{N_{i_j}(t)} \big| dt \Bigr)^2 \Bigr)^{1/2}  \\
      &\leq \int_0^\Lambda \Biggl( \sum_{j = 1}^{M-1} \sup_{k_j \in [i_j, i_{j+1})} \big| a_{N_{k_j}(t)} - a_{N_{i_j}(t)} \big|^2 \Biggr)^{1/2} dt.  \\
    \end{aligned}
  \end{equation}
  Due to the fact that $k \to \sum_{n = 1}^N \lambda_n^k$ is decreasing, one knows that $N_k(t)$ is non-decreasing, therefore the right-hand side of \eqref{weight-transference-for-operators-in-f-inequality-3} can be estimated by:
  \begin{equation}
    \int_0^\Lambda \Biggl( \sum_{j = 1}^{M-1} \sup_{k \in [N_{i_j}(t), N_{i_{j+1}}(t))} \big| a_k - a_{N_{i_j}(t)} \big|^2 \Biggr)^{1/2} dt.  \\
  \end{equation}
  There are countably many sequences of the form $(N_{i_1}(t), \ldots, N_{i_M}(t))$ when $t$ varies in the interval $[0, \Lambda]$, therefore \eqref{weight-transference-for-operators-in-f-inequality-2} is satisfied indeed.
  In a similar fashion to \eqref{weight-transference-for-operators-in-f-inequality-2}, we would like to verify now that:
  \begin{equation} \label{weight-transference-for-operators-in-f-inequality-4}
    \lambda N_\lambda^{1/2} \Bigl( \sum_{n = 1}^\infty \lambda_n^k a_n : k \in \NN \Bigr) \leq \sum_{j = 0}^\infty 10 \Lambda \frac{3^{j/2}}{2^j} \frac{2^j \lambda}{10 \Lambda} N_{\frac{2^j \lambda}{10 \Lambda}}^{1/2}( a_n : n \in \NN ).
  \end{equation}
  Take a sequence $(k_i)_{i = 0}^m$ so that for every $i \in \{0, \ldots, m-1\}$ one has:
  \begin{equation}
    \sum_{n = 1}^\infty \lambda_n^{k_{i+1}} a_n - \sum_{n = 1}^\infty \lambda_n^{k_i} a_n > \lambda
  \end{equation}
  In particular, we obtain that:
  \begin{equation} \label{weight-transference-for-operators-in-f-inequality-5}
    \sum_{i = 0}^{m-1} \int_0^\Lambda |a_{N_{k_{i+1}}(t)} - a_{N_{k_i}(t)}| dt > m \lambda,
  \end{equation}
  where we have used \eqref{weight-transference-identity}. Pick an arbitrary $t$, we claim that there exists a nonnegative integer $j$ so that:
  \begin{equation} \label{weight-transference-for-operators-in-f-inequality-6}
    u_j(t) := \bigg|\bigl\{ i \in \{0, \ldots, m-1 \} : |a_{N_{k_{i+1}}(t)} - a_{N_{k_i}(t)}| > \frac{2^j \lambda}{5 \Lambda} \bigr\} \bigg| \geq \frac{m}{3^j}
  \end{equation}
  If that is not the case then with $u_{-1}(t) := m$ one gets that:
  \begin{equation} \label{weight-transference-for-operators-in-f-inequality-7}
    \begin{aligned}
      \sum_{i = 0}^{m-1} \int_0^\Lambda |a_{N_{k_{i+1}}(t)} - a_{N_{k_i}(t)}| dt &\leq \int_0^\Lambda \sum_{j = 0}^\infty (u_{j-1}(t) - u_j(t)) \frac{2^j \lambda}{5 \Lambda} dt \\
      &\leq m \frac{\lambda}{5} + \int_0^\Lambda \sum_{j = 0}^\infty u_j(t) \frac{2^j \lambda}{5 \Lambda} dt. \\
    \end{aligned}
  \end{equation}
  Since we assume that $u_j(t)$ is everywhere bounded by $\frac{m}{3^j}$, by combining \eqref{weight-transference-for-operators-in-f-inequality-5} and \eqref{weight-transference-for-operators-in-f-inequality-7}, we derive the inequality:
  \begin{equation}
    m \lambda < \frac{m \lambda}{5} + m \lambda \sum_{j = 0}^\infty \frac{2^j}{5 \cdot 3^j} = \frac{4 m \lambda}{5}
  \end{equation}
  clearly giving contradiction, unless $m = 0$, in which case $N_\lambda = 0$. Now, due to \eqref{weight-transference-for-operators-in-f-inequality-6}, we obtain that:
  \begin{equation}
    N_{\frac{2^j \lambda}{10 \Lambda}}(a_n : n \in \NN) \geq \frac{m}{3^j}
  \end{equation}
  and also:
  \begin{equation}
    \lambda m^{1/2} \leq 10 \Lambda \frac{3^{j/2}}{2^j} \frac{2^j \lambda}{10 \Lambda} N_{\frac{2^j \lambda}{10 \Lambda}}^{1/2}( a_n : n \in \NN ). 
  \end{equation}
  Taking a supremum in $m$, one arrives at \eqref{weight-transference-for-operators-in-f-inequality-4}.
  Suppose now that the sequence $\Bigl( \frac{w_n'}{w_n} \Bigr)_{n \in \NN}$ increases monotonically. Then for:
  \begin{equation} \label{weight-transference-defined-sequence}
    \lambda_n^k := \begin{cases} \frac{W_n}{W_k'} (\frac{w_n'}{w_n} - \frac{w_{n+1}'}{w_{n+1}}) &\text{ if } n \in [1, k) \\
      \frac{W_k}{W_k'} \frac{w_k'}{w_k} &\text{ if } n = k \\
      0 &\text{ otherwise, }
    \end{cases}
  \end{equation}
  \begin{equation} \label{weight-transference-defined-sequence-2}
    A_N = \sum_{n = 1}^N w_n a_n, \qquad A_N' = \sum_{n = 1}^N w_n' a_n
  \end{equation}
  and also:
  \begin{equation}
    \tilde{\lambda}_n^k = \begin{cases} - \lambda_n^k &\text{ if } n \in [1, k) \\
      2C - \lambda_k^k &\text{ if } n = k \\
      0 &\text{ otherwise, }
    \end{cases}
  \end{equation}
  we see that $\tilde{\lambda}_n^k$ are nonnegative and by the partial summation:
  \begin{equation}
    A_k' = 2C A_k - \sum_{n = 1}^\infty \tilde{\lambda}_n^k A_n.
  \end{equation}
  Similarly to page 15 of \cite{mirek2017variational}, one sees that for any positive integer $N$: $k \to (\sum_{n = 1}^N \tilde{\lambda}_n^k)_{k \in \NN}$ is decreasing with:
  \begin{equation}
    \sum_{n = 1}^\infty \tilde{\lambda}_n^k = 2C - 1.
  \end{equation}
  Therefore, using \eqref{weight-transference-for-operators-in-f-inequality-2}, one gets that:
  \begin{equation}
    \begin{aligned}
    \text{Osc}_I(A_N' : N \in \NN) &\leq 2C \text{Osc}_I(A_N : N \in \NN) + \text{Osc}_I \Bigl( \sum_{n = 1}^k \tilde{\lambda}_n^k A_n : k \in \NN \Bigr) \\
    &\leq 2C \text{Osc}_I(A_N : N \in \NN) + \sum_j C_j \text{Osc}_{I_j} \Bigl( A_N : N \in \NN \Bigr) 
    \end{aligned}
  \end{equation}
  as we wanted. A simpler argument works with the jump counting operator (where we use \eqref{weight-transference-for-operators-in-f-inequality-4}) and when $(\frac{w_n'}{w_n})$ decreases monotonically, where we maintain the same definitions as \eqref{weight-transference-defined-sequence} and \eqref{weight-transference-defined-sequence-2}.
\end{proof}

The proof of Theorem \ref{variational-ergodic-theorem-special-prime-numbers} is almost the same as the proof of rapid convergence for polynomial ergodic averages given in Chapter 8 of \cite{krause2022discrete}. We suggest strongly that reader first look at what is happening there, before jumping into the proof below. We put emphasis on all the aspects where the reasoning goes differently, but we will unify the approach for all operators in the family $\mathcal{F}$:
\begin{proof}[Proof of Theorem \ref{variational-ergodic-theorem-special-prime-numbers}] We proceed in a sequence of steps.

  \textit{\textbf{Step 1:} Attaching the von Mangoldt weight to the ergodic average.} Fix an operator $U$ from the family $\mathcal{F}$. Similar to what we did for the maximal ergodic theorem, for the sake of circle method, we would like to investigate the analogous inequality for the von Mangoldt weighted averages. 
    Applying Proposition \ref{weight-transference-for-operators-in-f} with $w_n' = \log n$, $w_n = 1$ and $a_n = \mathbf{1}_\PP f(x - n)$, we reduce our goal to showing that:
    \begin{equation}
      \norm{U( A_N' f(x) : N \in \NN )}_{l^p(\ZZ)} \lesssim_{n, p} r(U) \norm{f}_{l^p(\ZZ)}.
    \end{equation}
    We will use the sequence of variables $\rho, A_0, \epsilon$ so that we fix the order:
    \begin{equation} \label{order-of-variables}
      \rho^{-1} \gg A_0 \gg B \gg A \gg \epsilon^{-1} \gg \theta_p \gg 1,
    \end{equation}
    with $A, B$ as in \ref{minor-arc-inequality} and $\theta_p$ suitable for interpolation arguments.

  \textit{\textbf{Step 2:} Sparsifying the sequence of indexes in the variational operator.} We use the splitting into long and short variation as in \eqref{variation-splitted-4}, in this concrete situation we take sequence $\mathcal{B} = (\lfloor 2^{k^\epsilon} \rfloor)_{k = 1}^\infty$. Then Theorem \ref{variational-ergodic-theorem-special-prime-numbers} will follow from the two inequalities:
    \begin{equation} \label{long-variational-ergodic-theorem-special-prime-numbers}
      \norm{U^L_{\mathcal{B}}( A_N' f(x) : N \in \NN )}_{l^p(\ZZ)} \lesssim_{\rho, n, p} r(U) \norm{f}_{l^p(\ZZ)}
    \end{equation}
    and
    \begin{equation}
      \norm{\mathcal{V}^{2, S}_{\mathcal{B}}( A_N' f(x) : N \in \NN )}_{l^p(\ZZ)} \lesssim_{\rho, n, p} \norm{f}_{l^p(\ZZ)}.
    \end{equation}
    The sequence $\mathcal{B}$ is sufficiently dense for showing the second statement:
    \begin{equation}
      \begin{aligned}
        &\norm{\mathcal{V}^{2, S}_{\mathcal{B}} ( A_N' f(x) : N \in \NN )}_{l^p(\ZZ)} \\
        &\leq \norm{\Biggl( \sum_{n = 1}^\infty \Bigl( \sup_{\lfloor 2^{n^\epsilon} \rfloor \leq n_1 < \ldots < n_t < \lfloor 2^{(n+1)^\epsilon} \rfloor} \sum_{k = 1}^{t-1} |A_{n_{k+1}}' f - A_{n_k}' f| \Bigr)^2 \Biggr)^{1/2}(x)}_{l^p(\ZZ)} \\
        &\leq \norm{\Biggl( \sum_{n = 1}^\infty \Bigl( \sum_{j = \lfloor 2^{n^\epsilon} \rfloor}^{\lfloor 2^{(n+1)^\epsilon} \rfloor - 1} |A_{j+ 1}' f - A_j' f| \Bigr)^2 \Biggr)^{1/2}(x)}_{l^p(\ZZ)} \\
        &\leq \Biggl( \sum_{n = 1}^\infty \Bigl( \sum_{j = \lfloor 2^{n^\epsilon} \rfloor}^{\lfloor 2^{(n+1)^\epsilon} \rfloor - 1} \norm{A_{j+ 1}' f - A_j' f}_{l^p(\ZZ)} \Bigr)^{\min(2, p)} \Biggr)^{1/\min(2, p)} \\
        &\leq \Biggl( \sum_{n = 1}^\infty \Bigl( \sum_{j = \lfloor 2^{n^\epsilon} \rfloor}^{\lfloor 2^{(n+1)^\epsilon} \rfloor - 1} \frac{1 + \log(j+1)}{j+1} \norm{f}_{l^p(\ZZ)} \Bigr)^{\min(2, p)} \Biggr)^{1/\min(2, p)} \\
        &\lesssim_p \Biggl( \sum_{n = 1}^\infty n^{\min(2, p)(2 \epsilon - 1)} \Biggr)^{1/2} \norm{f}_{l^p(\ZZ)} \lesssim \norm{f}_{l^p(\ZZ)} \\
      \end{aligned}
    \end{equation}
    From now on, we endeavour to establish \eqref{long-variational-ergodic-theorem-special-prime-numbers}.

  \textit{\textbf{Step 3:} Addressing the contribution from the minor arcs.}  First, we will reduce \eqref{long-variational-ergodic-theorem-special-prime-numbers} to the same statement for a sequence of operators precomposed with Ionescu-Wainger multipliers:
    \begin{equation} \label{long-variational-ergodic-theorem-special-prime-numbers-2}
      \norm{U \Biggl( A_{\lfloor 2^{k^\epsilon} \rfloor}' \Bigl(\Pi_{\leq k^{A_0}}(\mathbf{m})^\vee \ast f \Bigr)(x) : k \in \NN \Biggr)}_{l^p(\ZZ)},
    \end{equation}
    where the multiplier $\mathbf{m}$ is $\varphi_{d' k^{\epsilon}} (\beta)$ with $d' = \frac{2d - 1}{2}$, where $d$ is the degree of the polynomial $P$. 
    In both inequalities \eqref{long-variational-ergodic-theorem-special-prime-numbers} and \eqref{long-variational-ergodic-theorem-special-prime-numbers-2} one may restrict the $k$-range only to integers bigger than $O_\rho(1)$. The support of $\mathbf{m}$ is in the interval $[- 2^{- d' k^\epsilon}, 2^{- d' k^\epsilon}]$, which in turn is inside the interval $[- \exp(- k^{A_0 \rho}), \exp(- k^{A_0 \rho})]$ due to \eqref{order-of-variables}. This means that conditions of \eqref{ionescu-wainger-inequality} are satisfied for $\Pi_{\leq k^{A_0}}(\mathbf{m})$. In order to show the reduction stated earlier in this step, observe that:
    \begin{equation} \label{long-variational-ergodic-theorem-special-prime-numbers-2-split}
      \norm{A_{\lfloor 2^{k^\epsilon} \rfloor}' \Bigl(f - \Pi_{\leq k^{A_0}}(\mathbf{m})^\vee \ast f \Bigr)(x)}_{l^p(\ZZ)} \lesssim_{n, \rho, p} k^{-2} \norm{f}_{l^p(\ZZ)}
    \end{equation}
    as we can employ \eqref{quasi-triangle-inequality} together with \eqref{l1-control-on-operators-in-f}. Inequality \eqref{long-variational-ergodic-theorem-special-prime-numbers-2-split} is proved using interpolation methods, which we now describe. Due to \eqref{ionescu-wainger-inequality}, one has that:
    \begin{equation} \label{long-variational-ergodic-theorem-special-prime-numbers-2-lp}
      \norm{A_{\lfloor 2^{k^\epsilon} \rfloor}' \Bigl(f - \Pi_{\leq k^{A_0}}(\mathbf{m})^\vee \ast f \Bigr)(x)}_{l^p(\ZZ)} \lesssim_{\rho, p} \norm{f}_{l^p(\ZZ)}.
    \end{equation}
    On the other hand, the supremum of:
    \begin{equation}
      |\widehat{K_{\lfloor 2^{k^\epsilon} \rfloor}} \cdot (1 - \Pi_{\leq k^{A_0}}(\mathbf{m}))|
    \end{equation}
    on the torus is smaller than as the supremum of $\widehat{K_{\lfloor 2^{k^\epsilon} \rfloor}}$ outside of the union:
    \begin{equation}
      \mathcal{R}_{k^{A_0}} := \bigcup_{a/q : a \perp q, q \leq k^{A_0}} \Bigl[ \frac{a}{q} - 2^{-d' k^\epsilon - 2}, \frac{a}{q} + 2^{-d' k^\epsilon - 2} \Bigr].
    \end{equation}
    Therefore, due to \eqref{order-of-variables}, one gets that:
    \begin{equation}
      m_B \subset \TT \backslash \mathcal{R}_{k^{A_0}},
    \end{equation}
    so Theorem \ref{minor-arc-inequality} implies that:
    \begin{equation} \label{long-variational-ergodic-theorem-special-prime-numbers-2-l2}
      |\widehat{K_{\lfloor 2^{k^\epsilon} \rfloor}} \cdot (1 - \Pi_{\leq k^{A_0}}(\mathbf{m}))| \lesssim_n \frac{1}{k^{\epsilon A}} \leq \frac{1}{k^{\theta_p}}.
    \end{equation}
    Eventually, \eqref{long-variational-ergodic-theorem-special-prime-numbers-2-l2}, \eqref{long-variational-ergodic-theorem-special-prime-numbers-2-lp}, the choice of $\theta_p$, and Plancherel's identity imply \eqref{long-variational-ergodic-theorem-special-prime-numbers-2-split}, as we wanted.

  \textit{\textbf{Step 4:} Taking advantage of major arc behaviour given in Theorem \ref{major-arc-behaviour}.}  The next step requires passing from \eqref{long-variational-ergodic-theorem-special-prime-numbers-2} to:
    \begin{equation} \label{long-variational-ergodic-theorem-special-prime-numbers-3}
      \norm{U( (M_k^{(1)})^\vee \ast f(x) : k \in \NN )}_{l^p(\ZZ)} \lesssim_{n, p} r(U) \norm{f}_{l^p(\ZZ)},
    \end{equation}
    where:
    \begin{equation}
      M_k^{(1)}(\alpha) = \sum_{\mathbf{h}(a/q) \leq k^{A_0}} \frac{S(a, q)}{R_n \varphi_2(q_0)} v_{\lfloor 2^{k^\epsilon} \rfloor} \Bigl(\alpha - \frac{a}{q} \Bigr) \varphi_{d' k^{\epsilon}} (\alpha - \frac{a}{q}).
    \end{equation}
    Thanks to \ref{quasi-triangle-inequality} and \eqref{l1-control-on-operators-in-f}, it is sufficient for that reduction to get an appropriate bound for:
    \begin{equation} \label{long-variational-ergodic-theorem-special-prime-numbers-3-split}
      \norm{ \Bigl( (M_k^{(1)})^\vee - K_{\lfloor 2^{k^\epsilon} \rfloor} \ast \Pi_{\leq k^{A_0}}(\mathbf{m})^\vee \Bigr) \ast f(x)}_{l^p(\ZZ)}.
    \end{equation}
    Since the size of $k^{A_0}$-th Ionescu-Wainger set is $O(2^{k^{A_0 \rho}})$ and $v_{\lfloor 2^{k^\epsilon} \rfloor}^\vee$, $\varphi(2^{d' k^\epsilon})^\vee$ have $O(1)$-bounded $l^1$ norms, we get that above term is bounded by:
    \begin{equation}
      O_{\rho, p}(2^{k^{A_0 \rho}} \norm{f}_{l^p(\ZZ)}).
    \end{equation}
    We contrast this with $l^2(\ZZ)$-methods. Namely, the identity from Theorem \ref{major-arc-behaviour} implies that:
    \begin{equation}
      \norm{ \Bigl( (M_k^{(1)})^\vee - K_{\lfloor 2^{k^\epsilon} \rfloor} \ast \Pi_{\leq k^{A_0}}(\mathbf{m})^\vee \Bigr) \ast f(x)}_{l^2(\ZZ)} \lesssim_B \exp(-c k^{\epsilon/2}),
    \end{equation}
    so \eqref{order-of-variables} and interpolation implies that \eqref{long-variational-ergodic-theorem-special-prime-numbers-3-split} is at most $2^{- k^{\epsilon/3}}$, which is sufficient for this reduction.

  \textit{\textbf{Step 5:} Oscillatory integral estimates allows us to pass to the multiplier with support not depending on $k$.}  The following reduction uses again the aforementioned reference and is constructed so that the Ionescu-Wainger operator pops up:
    \begin{equation} \label{long-variational-ergodic-theorem-special-prime-numbers-4}
      \norm{U( (M_k^{(2)})^\vee \ast f(x) : k \in \NN )}_{l^p(\ZZ)} \lesssim_{n, p} r(U) \norm{f}_{l^p(\ZZ)},
    \end{equation}
    where:
    \begin{equation}
      M_k^{(2)}(\alpha) = \sum_{s : 2^s \leq k^{A_0}} \sum_{\mathbf{h}(a/q) = 2^s} \frac{S(a, q)}{R_n \varphi_2(q_0)} v_{\lfloor 2^{k^\epsilon} \rfloor} \Bigl(\alpha - \frac{a}{q} \Bigr) \varphi_{2^{d' \epsilon s / A_0}} (\alpha - \frac{a}{q}).
    \end{equation}
    Indeed, assuming \eqref{long-variational-ergodic-theorem-special-prime-numbers-4}, we may use \eqref{quasi-triangle-inequality} and \eqref{l1-control-on-operators-in-f} to show that: 
    \begin{equation} \label{long-variational-ergodic-theorem-special-prime-numbers-4-nonsplit}
      \sum_{k = 1}^\infty \norm{(M_k^{(1)} - M_k^{(2)})^\vee \ast f}_{l^p(\ZZ)} \lesssim_{n, p, \rho} \norm{f}_{l^p(\ZZ)}.
    \end{equation}
    We use strong $l^2$ estimates and acceptable $l^p$ estimates to establish \eqref{long-variational-ergodic-theorem-special-prime-numbers-4}.
    Note that $M_k^{(1)} - M_k^{(2)}$ is the same as:
    \begin{equation} \label{difference-of-ionescu-wainger-multipliers}
      \sum_{s : 2^s \leq k^{A_0}} \sum_{\mathbf{h}(a/q) = 2^s} \frac{S(a, q)}{R_n \varphi_2(q_0)} v_{\lfloor 2^{k^\epsilon} \rfloor} \Bigl(\alpha - \frac{a}{q} \Bigr) \Bigl( \varphi_{d' k^{\epsilon}} (\alpha - \frac{a}{q}) - \varphi_{2^{d' \epsilon s / A_0}} (\alpha - \frac{a}{q}) \Bigr).
    \end{equation}
    For the $l^p$ case, since $\frac{d' \epsilon}{A_0} > \rho$, and the third equality from Lemma \ref{arithmetic-functions-estimates} is satisfied, namely 
    \begin{equation}
      \frac{S(a, q)}{R_n \varphi_2(q_0)}
    \end{equation}
    is essentially a convex combination of phases. We conclude from \eqref{ionescu-wainger-inequality} that the above multiplier satisfies:
    \begin{equation} \label{long-variational-ergodic-theorem-special-prime-numbers-4-lp}
      \norm{(M_k^{(1)} - M_k^{(2)})^\vee \ast f}_{l^p(\ZZ)} \lesssim_{p, \rho} \norm{f}_{l^p(\ZZ)}.
    \end{equation}
    We now consider what is happening in the $l^2$-situation. More precisely, one has that for fixed $\alpha \in \TT$, there is only one term inside \eqref{difference-of-ionescu-wainger-multipliers} that does not vanish. If the accompanying fraction is $\frac{a}{q}$, then:
    \begin{equation}
      |\alpha - \frac{a}{q}| \gtrsim 2^{-d' k^\epsilon},
    \end{equation}
    so the oscillatory integral $v_{\lfloor 2^{k^\epsilon} \rfloor}(\alpha - \frac{a}{q})$ is bounded by $2^{- \frac{1}{2d} k^\epsilon}$ (due to Lemma B.2. from \cite{krause2022discrete}), which means that the expression \eqref{difference-of-ionescu-wainger-multipliers} is uniformly bounded by $2^{- \frac{1}{2d} k^\epsilon}$, so:
    \begin{equation}
      \norm{(M_k^{(1)} - M_k^{(2)})^\vee \ast f}_{l^2(\ZZ)} \lesssim 2^{- \frac{1}{2d} k^\epsilon} \norm{f}_{l^2(\ZZ)}.
    \end{equation}
    Interpolating this inequality together with \eqref{long-variational-ergodic-theorem-special-prime-numbers-3-split} followed by summing the result over all natural $k$ leads to \eqref{long-variational-ergodic-theorem-special-prime-numbers-4-nonsplit}.

  \textit{\textbf{Step 6:} Factorization of the multiplier into two parts: decoupling arithmetic structure and analytic structure.}  In order to establish \eqref{long-variational-ergodic-theorem-special-prime-numbers-4}, it suffices to estimate:
    \begin{equation} \label{long-variational-ergodic-theorem-special-prime-numbers-5}
      \sum_{s \geq 1} \norm{U( M_{k, s}^\vee \ast f(x) : k \geq 2^{s / A_0} )}_{l^p(\ZZ)} \lesssim_{n, p} r(U) \norm{f}_{l^p(\ZZ)},
    \end{equation}
    where we factor:
    \begin{equation} \label{long-variational-ergodic-theorem-factorization}
      M_{k, s} = T_s \Pi_s [v_{\lfloor 2^{k^\epsilon} \rfloor}(\cdot) \varphi_{2^{d' \epsilon s / A_0}}] 
    \end{equation}
    with
    \begin{equation} \label{arithmetic-multiplier}
      T_s(\alpha) = \sum_{\mathbf{h}(a/q) = 2^s} \frac{S(a, q)}{R_n \varphi_2(q_0)} \varphi_{2^{d' \epsilon s / A_0} + 1}
    \end{equation}
    In order to deduce \eqref{long-variational-ergodic-theorem-special-prime-numbers-4} from \eqref{long-variational-ergodic-theorem-special-prime-numbers-5}, just observe that:
    \begin{equation}
      M_k^{(2)} = \sum_{s : 2^s \leq k^{A_0}} M_{k, s}
    \end{equation}

  \textit{\textbf{Step 7:} Showing that Weyl sum estimates imply the exponential saving for the arithmetic factor.} We first discuss how to prove:
    \begin{equation} \label{long-variational-ergodic-theorem-weyl-and-mollifier}
      \norm{T_s}_{M^p(\ZZ)} \lesssim 2^{- \Theta_{d, p}(1) s}.
    \end{equation}
    Observe that by Lemma \ref{arithmetic-functions-estimates}, the $l^\infty$-norm of $T_s$ is of order $O(2^{- c_d s})$, therefore:
    \begin{equation} \label{arithmetic-multiplier-l2-situation}
      \norm{T_s}_{M^2(\ZZ)} \lesssim 2^{-c_d s}.
    \end{equation}
    At this point, using point \ref{ionescu-wainger-theory-recap-a} on the quantity of rationals with Ionescu-Wainger height leads us to the upper bound on $M^p(\ZZ)$-norm:
    \begin{equation} \label{arithmetic-multiplier-lp-situation}
      \norm{T_s}_{M^p(\ZZ)} \lesssim 2^{2^{\rho s}},
    \end{equation}
    which is too large to interpolate. Instead, we will apply our estimates for Weyl sums in \eqref{arithmetic-multiplier} to the multiplier \eqref{kernel-of-ergodic-average}:
    \begin{equation}
      \begin{aligned}
        T_s(\alpha) &= \sum_{\mathbf{h}(a/q) = 2^s} \widehat{K_{2^{2^{\epsilon s/2A_0}}}}(\alpha) \varphi_{2^{d' \epsilon s / A_0} + 1} (\alpha - \frac{a}{q}) \\
        &+ \sum_{\mathbf{h}(a/q) = 2^s} \Biggl(\frac{S(a, q)}{R_n \varphi_2(q_0)} - \widehat{K_{2^{2^{\epsilon s/2A_0}}}}(\alpha) \Biggr) \varphi_{2^{d' \epsilon s / A_0} + 1} (\alpha - \frac{a}{q}) \\
        &=: T_s^{(1)}(\alpha) + T_s^{(2)}(\alpha),
      \end{aligned}
    \end{equation}
    and will try to prove that these two operators $T_s^{(1)}$ and $T_s^{(2)}$ have respectively exponentially and double exponentially decaying $M^p(\ZZ)$-norms. Regarding $T_s^{(1)}$, the $M^p(\ZZ)$-norm of it is bounded by $O_{p, \rho}(1)$ (using \eqref{ionescu-wainger-inequality}) whereas $M^2(\ZZ)$-norm still obeys the exponential decay (using a combination of Lemma \ref{arithmetic-functions-estimates}, \eqref{special-prime-kernel-common-thing} and Property \ref{ionescu-wainger-theory-recap-b} of the Ionescu-Wainger theory recap).
    As for $T_s^{(2)}$, observe that it has Fourier support inside intervals of radius at most $2^{-2^{\epsilon s/A_0}}$ around rationals of Ionescu-Wainger height equal $2^s$. So, we have the following estimation for the $M^2(\ZZ)$-norm coming from the inequality:
    \begin{equation}
      \begin{aligned}
        \bigg| \frac{S(a, q)}{R_n \varphi_2(q_0)} - \widehat{K_{2^{2^{\epsilon s/2A_0}}}}(\alpha) \bigg| &\leq \bigg| \frac{S(a, q)}{R_n \varphi_2(q_0)} - \widehat{K_{2^{2^{\epsilon s/2A_0}}}} \Bigl( \frac{a}{q} \Bigr) \bigg| \\
        &+ \bigg| \widehat{K_{2^{2^{\epsilon s/2A_0}}}} \Bigl(\frac{a}{q} \Bigr) - \widehat{K_{2^{2^{\epsilon s/2A_0}}}}(\alpha) \bigg|. \\
      \end{aligned}
    \end{equation}
    From the Lipschitz nature of $\widehat{K_{2^{2^{\epsilon s/2A_0}}}}$, one gets that the second magnitude is at most $2^{-2^{\epsilon s/4 A_0}}$ whereas the first one is of order $O(2^{-2^{\epsilon s/4 A_0}})$ due to \eqref{special-prime-kernel-common-thing}. So:
    \begin{equation}
      \norm{T_s^{(2)}}_{M^2(\ZZ)} \lesssim 2^{-2^{\epsilon s/4A_0}}.
    \end{equation}
    On the other hand, from point \ref{ionescu-wainger-theory-recap-a} of Ionescu-Wainger theory introduction, one knows that number of rationals modulo $1$ that have Ionescu-Wainger height $2^s$ is not bigger than $O(2^{2^{\rho s}})$, so:
    \begin{equation}
      \norm{T_s^{(2)}}_{M^p(\ZZ)} \lesssim_\rho 2^{2^{s \rho}}.
    \end{equation}
    Interpolating last two inequalities lead us to the bound:
    \begin{equation}
      \norm{T_s^{(2)}}_{M^p(\ZZ)} \lesssim_{\rho, p, A_0} 2^{-2^{\epsilon s/10 A_0}},
    \end{equation}
    which is what we wanted.

  \textit{\textbf{Step 8:} Distinguishing large and small scales in the continuous factor.}  The final goal for us is to establish the inequality:
    \begin{equation} \label{long-variational-ergodic-theorem-only-analytic-ionescu-wainger}
      \norm{U \Bigl( \Pi_s [v_{\lfloor 2^{k^\epsilon} \rfloor}(\cdot) \varphi_{2^{d' \epsilon s / A_0}}]^\vee \ast f(x) : k \geq 2^{s / A_0} \Bigr) }_{l^p(\ZZ)} \lesssim r(U) 2^{3 \rho s} \norm{f}_{l^p(\ZZ)}.
    \end{equation}
    Indeed, combining \eqref{long-variational-ergodic-theorem-factorization}, \eqref{long-variational-ergodic-theorem-weyl-and-mollifier} and \eqref{long-variational-ergodic-theorem-only-analytic-ionescu-wainger} would exhibit exponential decay with respect to $s$:
    \begin{equation}
      \begin{aligned}
        \norm{U \Bigl( M_{k, s}^\vee \ast f(x) : k \geq 2^{s / A_0} \Bigr) }_{l^p(\ZZ)} &\leq \norm{U \Bigl( \Pi_s [v_{\lfloor 2^{k^\epsilon} \rfloor}(\cdot) \varphi_{2^{d' \epsilon s / A_0}}]^\vee \ast T_s^\vee \ast f(x) : k \geq 2^{s / A_0} \Bigr) }_{l^p(\ZZ)}  \\
        &\lesssim r(U) 2^{3 \rho s} \norm{T_s^\vee \ast f}_{l^p(\ZZ)} \\
        &\lesssim r(U) 2^{3 \rho s - c_p s} \norm{f}_{l^p(\ZZ)}
      \end{aligned}
    \end{equation}
    as we know that $\frac{1}{\rho} > O_p(1)$. 
    We split the range in the variation \eqref{long-variational-ergodic-theorem-only-analytic-ionescu-wainger} into a small scale part:
    \begin{equation} \label{only-analytic-ionescu-wainger-small-range}
      \norm{U \Bigl( \Pi_s [v_{\lfloor 2^{k^\epsilon} \rfloor}(\cdot) \varphi_{2^{d' \epsilon s / A_0}}]^\vee \ast f(x) : 2^{2^{3 s \rho}} \geq k \geq 2^{s / A_0} \Bigr) }_{l^p(\ZZ)} \lesssim 2^{3 \rho s} \norm{f}_{l^p(\ZZ)},
    \end{equation}
    and large scale one:
    \begin{equation} \label{only-analytic-ionescu-wainger-large-range}
      \norm{U \Bigl( \Pi_s [v_{\lfloor 2^{k^\epsilon} \rfloor}(\cdot) \varphi_{2^{d' \epsilon s / A_0}}]^\vee \ast f(x) : k \geq 2^{2^{3 s \rho}} \Bigr) }_{l^p(\ZZ)} \lesssim r(U) \norm{f}_{l^p(\ZZ)}.
    \end{equation}
    

  \textit{\textbf{Step 9:} The Rademacher-Menshov is used to verifying the small scale case.}
    Combining \eqref{rademacher-menshov-inequality} together with \eqref{ionescu-wainger-inequality}, it is enough to show that for every $i \leq 2^{3 s \rho}$ one has:
    \begin{equation}
      \norm{ \Bigl( \sum_{j = 0}^{\infty} \Bigg| \sum_{k \in I_j^i} ( (v_{\lfloor 2^{k^\epsilon} \rfloor} - v_{\lfloor 2^{(k+1)^\epsilon} \rfloor}) \varphi_{2^{d' \epsilon s / A_0}})^\vee \ast f \Bigg|^2 \Bigr)^{1/2} }_{L^p(\RR)} \lesssim \norm{f}_{L^p(\RR)}
    \end{equation}
    where $I_j^i = [j \cdot 2^i, (j+1) 2^i)$. That will naturally follow from a different inequality: suppose that one has an increasing sequence of real numbers $(t_i)_{i = 1}^\infty$, then we endeavour to show that:
    \begin{equation} \label{only-analytic-ionescu-wainger-small-range-2}
      \norm{ \Bigl( \sum_{i = 0}^{\infty} \big| ( (v_{t_{i+1}} - v_{t_i}) )^\vee \ast f \big|^2 \Bigr)^{1/2} }_{L^p(\RR)} \lesssim \norm{f}_{L^p(\RR)}
    \end{equation}
    The left-hand side above can be treated as an $\mathcal{V}^2$-variation along the range of all positive real numbers. Using \eqref{variation-splitted} for $r = 2$, one may reduce \eqref{only-analytic-ionescu-wainger-small-range-2} into showing $2$ identities below:
    \begin{equation} \label{only-analytic-ionescu-wainger-long-variation}
      \norm{\sum_{i \in \ZZ} \epsilon_i (v_{2^{i+1}} - v_{2^i})^\vee \ast f}_{L^p(\RR)} \lesssim_p \norm{f}_{L^p(\RR)}
    \end{equation} 
    for any sequence of complex numbers $(\epsilon_i)_{i \in \ZZ}$ bounded universally by $1$, by randomization, and:
    \begin{equation} \label{only-analytic-ionescu-wainger-short-variation}
      \norm{ \Bigl( \sum_{i \in \ZZ} \mathcal{V}^2(v_t^\vee \ast f : t \in [2^i, 2^{i+1}])^2 \Bigr)^{1/2}}_{L^p(\RR)} \lesssim_p \norm{f}_{L^p(\RR)}.
    \end{equation}
    Recall at the beginning that action of $v_t^\vee$ is given by the formula:
    \begin{equation}
      v_t^\vee \ast f(x) = \frac{1}{t} \int_0^t f(x - P(u)) du,
    \end{equation}
    therefore by the triangle inequality:
    \begin{equation}
      \norm{v_t^\vee \ast f}_{L^p(\RR)} \leq \norm{f}_{L^p(\RR)}
    \end{equation}
    for all measurable $f$, $p \geq 1$ and $t > 0$. In order to justify \eqref{only-analytic-ionescu-wainger-long-variation}, we would like to employ Theorem 2.28 from \cite{mirek2020bootstrapping} (particularly the second case). Take $B_i = v_{2^{i+1}} - v_{2^i}$ and let $S_i$ be the operator given by Fourier multiplier whose symbol is a smooth indicator of the set $[-2^{-di}, -2^{-d(i+1)}] \cup [2^{-d(i+1)}, 2^{-di}]$. Then the sequence $a_j$ constructed by inequality $(2.31)$ from \cite{mirek2020bootstrapping} is of size $O(2^{-j})$ so in particular the constant $\mathbf{a}$ from their Theorem $2.14$ is finite. 
    When $p \in (1, 2)$, then take $(q_0, q_1) = (1, p)$, then the $L^{q_1}$-boundedness of the operator $B_\ast$ follows from the Hardy-Littlewood inequality and the positivity of the function $v_t^\ast$. This means that all conditions of Theorem 2.28 are satisfied, therefore we get that \eqref{only-analytic-ionescu-wainger-long-variation} is indeed true. In the case $p \geq 2$, one refers to a duality argument. Let us now explain why \eqref{only-analytic-ionescu-wainger-short-variation} is true as well. This time we want to use Theorem 2.39 from \cite{mirek2020bootstrapping}. One sees that $f \to (v_{t + h}^\vee - v_t^\vee) \ast f$ is an operator from $L^1 \to L^1$ with the norm $O(\frac{h}{t})$. One can verify that by applying the triangle inequality on expression:
    \begin{equation}
      (v_{t + h}^\vee \ast f - v_t^\vee \ast f)(x) = \frac{h}{t(t+h)} \int_0^{t} f(x - P(u)) du - \frac{1}{t+h} \int_t^{t + h} f(x - P(u)) du.
    \end{equation}
    Therefore the operator norm satisfies a H\"older-type condition, so the conditions of Theorem $2.39$ is satisfied. By the positivity of the operator $v_t^\vee \ast$ and the Hardy-Littlewood inequality, one obtains the $(2.47)$ from \cite{mirek2020bootstrapping}.
    Now, fix $1 = q_0 < q_1 = p \leq 2$; in order to prove \eqref{only-analytic-ionescu-wainger-long-variation} it is enough to verify that the constant $\mathbf{a}$ that appears in the statement of Theorem 2.39 is finite. Assuming that $(S_i)_{i \in \ZZ}$ are taken the same as earlier, the short discussion after this theorem shows that $\mathbf{a}$ can be taken to be:
    \begin{equation}
      \sum_{l = 0}^\infty \sum_{j \in \ZZ} 2^{-\frac{(q_1 - 1) l}{2}} \min(1, 2^l 2^{-|j|/d})^{q_1 - 1} 
    \end{equation}
    which is naturally finite.

  \textit{\textbf{Step 10:} L\'epingle's inequality resolves the large scale case.} At last, we are left with showing that \eqref{only-analytic-ionescu-wainger-large-range} is satisfied. The trick is to pass from the multiplier $\Pi_s [v_{k^\epsilon}(\cdot) \varphi(2^{2^{\epsilon s / A_0}} \cdot)]$ to the product:
    \begin{equation}
      \sum_{k \geq 2^{2^{3s \rho}}} \norm{ \Pi_s [v_{\lfloor 2^{k^\epsilon} \rfloor}(\cdot) \varphi_{2^{d' \epsilon s / A_0}}](\beta) - \Bigl( \sum_{a \in [Q_s]} (v_{\lfloor 2^{k^\epsilon} \rfloor} \cdot \Phi_s)(\beta - \frac{a}{Q_s}) \Bigr) \Pi_s[\varphi_{2^{d' \epsilon s / A_0}}](\beta) }_{M^p(\ZZ)} \lesssim 1
    \end{equation}
    where $\Phi_s$ is a smooth mollifier adapted to the interval $[-\frac{1}{10Q_s}, \frac{1}{10 Q_s}]$. Since we are in the large scale, the supports of all summands inside $\Bigl( \sum_{a \in [Q_s]} (v_{\lfloor 2^{k^\epsilon} \rfloor} \cdot \Phi_s)(\beta - \frac{a}{Q_s}) \Bigr)$ or $\Pi_s[\varphi_{2^{d' \epsilon s / A_0}}](\beta)$ are disjoint, which together with a convexity argument in physical space justifies the identity above.
    So, Proposition \ref{sampling-principle-magyar-stein-wainger} together with \eqref{quasi-triangle-inequality} and \eqref{l1-control-on-operators-in-f} allows to bound:
    \begin{equation}
      \norm{U \Biggl( \Bigl( \sum_{a \in [Q_s]} (v_{\lfloor 2^{k^\epsilon} \rfloor} \cdot \Phi_s)(\beta - \frac{a}{Q_s}) \Bigr)^\vee \ast f(x) : k \geq 2^{2^{3 s \rho}} \Biggr)}_{l^p(\ZZ)} \lesssim r(U) \norm{f}_{l^p(\ZZ)}
    \end{equation}
    by reducing the above inequality to:
    \begin{equation}
      \norm{U \bigl( v_{\lfloor 2^{k^\epsilon} \rfloor}^\vee \ast f(x) : k \geq 2^{2^{3 s \rho}} \bigr)}_{L^p(\RR)} \lesssim r(U) \norm{f}_{L^p(\RR)}.
    \end{equation}
    However one may easily see that $v_t^\vee \ast f$ is the same as $M_t(f)$, see \eqref{lepingle-inequality-radon-averaging-operator}. Therefore the above inequality is satisfied and \eqref{ionescu-wainger-inequality} is responsible for handling the remaining factor $\Pi_s[\varphi_{2^{d' \epsilon s / A_0}}](\beta)$. This concludes our $l^p$-theory.
\end{proof}

\section{Divergence in $L^1$} \label{section-8}

Due to Sawyer's principle, we want to show that for arbitrary $C$, one has a measure preserving system $(X, \mu, T)$ and function $f \in L^1(X)$, so that:
\begin{equation}
  \norm{\sup_{M \in \NN} \bigg| \frac{1}{|\PP_n \cap [M]|} \sum_{p \in \PP_n \cap [M]} f(T^p x) \bigg|}_{L^{1, \infty}(X)} > C \norm{f}_{L^1}.
\end{equation}
We want to use the aforementioned Theorem 3.1. of \cite{lavictoire2011universally}. For this reason, let us justify that conditions $(3.1)-(3.5)$ from LaVictoire's work are satisfied, this will be sufficient to show that the investigated sequence does not satisfy the $L^1 \to L^{1, \infty}$ bounds. In the first two cases one needs to provide two sequences $(p_j)_{j \in \NN}$ and $(q_j)_{j \in \NN}$ so that density of residues coming from the sequence of primes of form $x^2 + ny^2$ tends with $j \to \infty$ respectively to $0$ and $1$, respectively. Conditions $(3.3)$ and $(3.4)$ are the most technical, they more or less describe how residues modulo $q_j$ and their sums are separated when $j \to \infty$. Nevertheless, it is known from the work of LaVictoire that when the limiting behaviour of the differences between consecutive residues is Poisson, then one gets $(3.3)$ and $(3.4)$ for free. Finally, the last condition says that all residues that are hit by $\PP_n$ are hit sufficiently often.

So, pick $(p_j)_{j \in \NN}$ to be an increasing sequence of primes such that for all $j \in \NN$: $\Bigl( \frac{-n}{p_j} \Bigr) = 1$, and choose $(q_j)_{j \in \NN}$ to be highly composite so that:
\begin{enumerate}
  \item For every $j$, $q_j$ is odd, squarefree and $-n$ is a quadratic residue that is coprime with respect to every prime divisor of $q_j$;
  \item The sequence $q_j^{-1} \varphi(q_j)$ converges to $0$ as $j \to \infty$.
\end{enumerate} 
Next, for every $Q$ so that all prime divisors $p$ of $Q$ satisfy $\Bigl( \frac{-n}{p} \Bigr) = 1$, let $\Lambda_Q$ represent the set of coprime residues with respect to $Q$.  For $(3.1)$ we use the fact that:
\begin{equation}
  \prod_{p \in \PP : (\frac{-n}{p}) = 1} \frac{p - 1}{p} = 0,
\end{equation}
due to the Chebotarev density theorem for $\QQ(\sqrt{-n})$. $(3.2)$ is coming from the convergence $\frac{\varphi(p)}{p} \to 1$ for primes, meanwhile $(3.3)$ and $(3.4)$ follow from the fact that $\PP(\Lambda_{q_j}^\gamma)$ converges to the Poisson distribution with respect to $\gamma$ as $j \to \infty$ (see Lemma $3.2.$ and the equation above it in \cite{lavictoire2011universally}). Originally, this fact was established by Hooley in \cite{hooley1965difference}, we mimic LaVictoire there in applying Hooley's work. Ultimately, the only thing that requires justification is the inequality:
\begin{equation} \label{nontrivial-condition-lavictoire}
  \lim_{m \to \infty} \frac{1}{|\PP_n \cap [m]|} \bigg| \{ p \in \PP_n \cap [m] : p \equiv a \pmod Q \} \bigg| \gtrsim \frac{1}{|\Lambda_Q|}
\end{equation}
whenever $Q$ is a product of distinct elements from the sequences $(p_j)_{j \in \NN}$ and $(q_j)_{j \in \NN}$. 
A standard partial summation argument allows us to obtain \eqref{prime-ideals-vs-prime-passing} from the following:
\begin{equation}
  \sum_{p \in \PP_n \cap [m]} e\Bigl(\frac{a p}{q}\Bigr) = S(a, q) \frac{\text{li}(m)}{R_n \varphi_2(q_0)} + O (m \exp(-c \sqrt{\log m})), \\
\end{equation}
where $\text{li}(m)$ represents the logarithmic integral:
\begin{equation}
  \int_2^m \frac{1}{\log t} dt.
\end{equation}
Using Fourier inversion, we compute the size of the set:
\begin{equation}
  \begin{aligned}
    \bigg| \{p \in \PP_n \cap [m] : p \equiv b \pmod Q \} \bigg| &= \frac{1}{Q} \sum_{p \in \PP_n \cap [m]} \sum_{a \pmod Q} e\Bigl(\frac{a (p - b)}{q}\Bigr)  \\
    &= \frac{1}{Q} \sum_{a \pmod Q} e\Bigl(\frac{-ba}{q}\Bigr) \sum_{p \in \PP_n \cap [m]} e\Bigl(\frac{ap}{q}\Bigr)  \\
    &= \frac{1}{Q} \sum_{a \pmod Q} S(a, q) e\Bigl(\frac{-ba}{q}\Bigr) \frac{\text{li}(m)}{R_n \varphi_2(q_0)} \\
    &+ O (m \exp(-c \sqrt{\log m})).
  \end{aligned}
\end{equation}
The later average: 
\begin{equation}
  \frac{1}{Q} \sum_{a \pmod Q} S(a, q) e\Bigl(\frac{-ba}{q}\Bigr)
\end{equation}
is the same as
\begin{equation}
  \bigg| \{ x, y \in [Q] : x^2 + ny^2 \perp Q, x^2 + ny^2 \equiv b \pmod Q \} \bigg|.
\end{equation}
If we manage to show that the value of this expression does not depend on $b$, then we will have shown \eqref{nontrivial-condition-lavictoire}. Since $Q$ is product of distinct prime factors and our claim has a local nature (due to the Chinese remainder theorem), one need only show that for every odd prime $p$ with $\Bigl( \frac{-n}{p} \Bigr) = 1$, the quantity:
\begin{equation} \label{lavictorie-last-condition-quantity}
  \bigg| \{ x, y \in [p] : x^2 + ny^2 \perp p, x^2 + ny^2 \equiv b \pmod p \} \bigg|
\end{equation}
does not depend on $b$, provided $b$ is coprime to $p$. So, suppose that $-n \equiv w^2 \pmod p$; we want to solve the congruence:
\begin{equation}\label{lavictorie-last-condition-quantity-2}
  (x - wy)(x + wy) \equiv b \pmod p.
\end{equation}
It is straightforward to see that the number of solutions \eqref{lavictorie-last-condition-quantity-2} and thus \eqref{lavictorie-last-condition-quantity} is precisely $p - 1$. That means we have checked all the conditions to apply Theorem \ref{divergence-theorem-special-prime-numbers} at LaVictoire \cite{lavictoire2011universally}.

\appendix

\section{Introduction to algebraic number theory} \label{appendix}

In this appendix we provide a little dictionary of terms coming from algebraic number theory, which we believe will help to understand the paper for those whose background is more analytic:
\begin{enumerate}
  \item The number field $K$ is an extension of $\QQ$ so that $\dim_{\QQ} K$ is finite. Classical examples of number fields are $\QQ(\sqrt{-l})$ or $\QQ(\exp(\frac{2 \pi i}{l}))$ for an integer $l$.
  \item An element $a \in K$ is called \textit{integral} if it is a root of a monic polynomial with integer coefficients. One can prove that the set of integral elements of field $K$ is closed under addition, subtraction and multiplication; we will denoted this ring by $\mathcal{O}_K$.
  \item For a commutative ring $R$ with identity, we define the $R$-module $M$ to be an abelian group with respect to addition, equipped with multiplication $\cdot : R \times M \to M$ so that for all $r, s \in R$ and $x, y \in M$:
  \begin{equation}
    r \cdot (x + y) = r \cdot x + r \cdot y
  \end{equation}
  \begin{equation}
    (r + s) \cdot x = r \cdot x + s \cdot x
  \end{equation}
  \begin{equation}
    rs \cdot x = r \cdot (s \cdot x)
  \end{equation}
  \begin{equation}
    1 \cdot x = x
  \end{equation}
  \item \textit{A fractional ideal} is an $\mathcal{O}_K$-submodule $J$ of $K$ such that there exists a non-zero $r \in \mathcal{O}_K$ with $rJ \subset \mathcal{O}_K$. We will denote the set of these ideals by $I_K$. 
  \item \textit{A prime ideal} is an ideal $\mathfrak{p} \subsetneq \mathcal{O}_K$ so that for any two $a, b \in \mathcal{O}_K$ with $ab \in \mathfrak{p}$, either $a$ or $b$ is an element of $\mathfrak{p}$.
  \item The product of two fractional ideals $I$ and $J$ is defined as:
  \begin{equation}
    IJ = \{ a_1 b_1 + \ldots + a_k b_k : \forall_i a_i \in I, b_i \in J \}.
  \end{equation}
  \item By the above, the set of fractional ideals form a commutative semigroup; in our case, we are dealing with rings $\mathcal{O}_K$, so for every element $I$ in $I_{\mathcal{O}_K}$ one can find an inverse ideal $J$ so that $IJ = \mathcal{O}_K$.
  \item An important theorem in algebraic number theory says that ideals in $\mathcal{O}_K$ obey \textit{unique factorization}, so that every ideal can be written as a product of prime ideals. Furthermore, every fractional ideal can be written uniquely as a product of prime ideals and their inverses.
  This gives an analogy between the arithmetic situation for integers: prime ideals are to rings $\mathcal{O}_K$ the analogues of primes in $\ZZ$, ideals correspond to integers and fractional ideals correspond to rational numbers (up to sign).
  \item The norm of an ideal $\mfn \subset \mathcal{O}_K$ is simply the size of quotient group $\mathcal{O}_K / \mfn$. We will denote it by $N \mfn$. This norm respects product $N (\mfn \mathfrak{k}) = N \mfn N \mathfrak{k}$, therefore one can extend this definition to fractional ideals. Moreover for $a \in K \backslash \{ 0 \}$, we define $Na = N((a))$, i.e. the norm of the element $a$ is the norm of the principal ideal generated by $a$.
  \item We will say that two fractional ideals are \textit{coprime} when there is no joint prime ideal factor in the corresponding factorizations.
  \item For a general number field $K$, ideals $I \subset \mathcal{O}_K$ may not be \textit{principal} (i.e.~generated by single element). The set of ideals are closed with respect to actions of multiplication and taking inverses, so the group of these ideals will be called $\mathcal{P}_K$. One of the main results in algebraic number theory says that the quotient group $I_K / \mathcal{P}_K =: \text{Cl}(K)$ is finite.
  \item Last but not least, by a Hecke character we will mean a function $\chi: I_K \to S^1$, and so that $\chi(\mfa \mfb) = \chi(\mfa) \chi(\mfb)$ for all $\mfa$ and $\mfb$ in $I_K$ so that for some $q \in \ZZ_+$, $\chi$ maps the set of principal ideals with generators congruent to $1$ modulo $q$ (denote it $\mathcal{P}_K^+(q)$) to $1$.
\end{enumerate}

\printbibliography
\end{document}